\documentclass[11pt]{article}
\usepackage{srcltx}
\usepackage{eurosym}
\usepackage{mathtools}
\usepackage{amsmath}
\usepackage{amsfonts}
\usepackage{amssymb}
\usepackage{comment}
\usepackage{amsthm}
\usepackage{graphicx}
\usepackage{mathrsfs}
\usepackage{xcolor}
\usepackage{exscale}
\usepackage{color, soul}
\usepackage{latexsym}
\usepackage{authblk}
\allowdisplaybreaks
\usepackage[colorlinks,plainpages=true,pdfpagelabels,hypertexnames=true,colorlinks=true,pdfstartview=FitV,linkcolor=blue,citecolor=red,urlcolor=black]{hyperref}
\PassOptionsToPackage{unicode}{hyperref}
\PassOptionsToPackage{naturalnames}{hyperref}
\usepackage{enumerate}
\usepackage[shortlabels]{enumitem}
\usepackage{bookmark}
\usepackage{wasysym}
\usepackage{esint}
\usepackage[ddmmyyyy]{datetime}
\usepackage[margin=1in]{geometry}
\makeatletter
\g@addto@macro\normalsize{%
	\setlength\abovedisplayskip{4pt}
	\setlength\belowdisplayskip{4pt}
	\setlength\abovedisplayshortskip{4pt}
	\setlength\belowdisplayshortskip{4pt}
}
\numberwithin{equation}{section}
\everymath{\displaystyle}
\usepackage[capitalize,nameinlink]{cleveref}
\crefname{section}{Section}{Sections}
\crefname{subsection}{Subsection}{Subsections}
\crefname{condition}{Condition}{Conditions}
\crefname{hypothesis}{Hypothesis}{Conditions}
\crefname{assumption}{Assumption}{Assumptions}
\crefname{lemma}{Lemma}{Lemmas}
\crefname{definition}{Definition}{Definitions}

\crefformat{equation}{\textup{#2(#1)#3}}
\crefrangeformat{equation}{\textup{#3(#1)#4--#5(#2)#6}}
\crefmultiformat{equation}{\textup{#2(#1)#3}}{ and \textup{#2(#1)#3}}
{, \textup{#2(#1)#3}}{, and \textup{#2(#1)#3}}
\crefrangemultiformat{equation}{\textup{#3(#1)#4--#5(#2)#6}}%
{ and \textup{#3(#1)#4--#5(#2)#6}}{, \textup{#3(#1)#4--#5(#2)#6}}%
{, and \textup{#3(#1)#4--#5(#2)#6}}

\Crefformat{equation}{#2Equation~\textup{(#1)}#3}
\Crefrangeformat{equation}{Equations~\textup{#3(#1)#4--#5(#2)#6}}
\Crefmultiformat{equation}{Equations~\textup{#2(#1)#3}}{ and \textup{#2(#1)#3}}
{, \textup{#2(#1)#3}}{, and \textup{#2(#1)#3}}
\Crefrangemultiformat{equation}{Equations~\textup{#3(#1)#4--#5(#2)#6}}%
{ and \textup{#3(#1)#4--#5(#2)#6}}{, \textup{#3(#1)#4--#5(#2)#6}}%
{, and \textup{#3(#1)#4--#5(#2)#6}}

\crefdefaultlabelformat{#2\textup{#1}#3}

\newtheorem{theorem} {Theorem}[section]

\newtheorem{lemma}[theorem]{Lemma}

\newtheorem{counter example}[theorem]{Counter Example}
\newtheorem{remark}[theorem] {Remark}

\def\CC{{\rm \kern.24em \vrule width.02em height1.4ex depth-.05ex \kern-.26emC}}

\def\TagOnRight

\def\AA{{it I} \hskip-3pt{\tt A}}

\def\QQ{\rlap {\raise 0.4ex \hbox{$\scriptscriptstyle |$}} {\hskip -0.1em Q}}

\makeatletter
\newcommand{\vo}{\vec{o}\@ifnextchar{^}{\,}{}}
\makeatother

\def\YYint#1#2#3{{\setbox0=\hbox{$#1{#2#3}{\iint}$}
		\vcenter{\hbox{$#2#3$}}\kern-.50\wd0}}

\def\XXint#1#2#3{{\setbox0=\hbox{$#1{#2#3}{\int}$}
		\vcenter{\hbox{$#2#3$}}\kern-.50\wd0}}

\makeatletter
\def\namedlabel#1#2{\begingroup
	\def\@currentlabel{#2}%
	\label{#1}\endgroup
}
\makeatother
\makeatletter
\newcommand{\rmh}[1]{\mathpalette{\raisem@th{#1}}}
\newcommand{\raisem@th}[3]{\hspace*{-1pt}\raisebox{#1}{$#2#3$}}
\makeatother

\newcounter{desccount}

\newcommand{\descref}[2]{\hyperref[#1]{\textnormal{\textcolor{black}{}\textcolor{blue}{ #2}\textcolor{black}{}}}}

\newcommand{\dref}[2]{\hyperref[#1]{\textcolor{black}{(}\textcolor{blue}{\bf #2}\textcolor{black}{)}}}
\newcommand{\be} {\begin{eqnarray}}
	\newcommand{\ee} {\end{eqnarray}}
\newcommand{\Bea} {\begin{eqnarray*}}
	\newcommand{\Eea} {\end{eqnarray*}}

\newcounter{whitney}
\refstepcounter{whitney}

\newcounter{ineqcounter}
\refstepcounter{ineqcounter}
\makeatletter
\def\ps@pprintTitle{%
	\let\@oddhead\@empty
	\let\@evenhead\@empty
	\def\@oddfoot{}%
	\let\@evenfoot\@oddfoot}
\makeatother
\usepackage[doublespacing]{setspace}
\usepackage[titletoc,toc,page]{appendix}

\makeatletter
\newcommand{\refcheckize}[1]{%
	\expandafter\let\csname @@\string#1\endcsname#1%
	\expandafter\DeclareRobustCommand\csname relax\string#1\endcsname[1]{%
		\csname @@\string#1\endcsname{##1}\wrtusdrf{##1}}%
	\expandafter\let\expandafter#1\csname relax\string#1\endcsname
}
\makeatother

\refcheckize{\cref}
\refcheckize{\Cref}


\makeatletter
\newcommand{\mainsectionstyle}{%
	\renewcommand{\@secnumfont}{\bfseries}
	\renewcommand\section{\@startsection{section}{2}%
		\z@{.5\linespacing\@plus.7\linespacing}{-.5em}%
		{\normalfont\bfseries}}%
}
\makeatother
\usepackage{pgf,tikz}
\usetikzlibrary{arrows}
\usetikzlibrary{decorations.pathreplacing}

\usepackage{xpatch}
\xpatchcmd{\MaketitleBox}{\hrule}{}{}{}
\xpatchcmd{\MaketitleBox}{\hrule}{}{}{}

\date{}

\usepackage{scalerel}

\makeatletter

\usepackage{subcaption}
\usepackage{hyperref}
\usepackage{caption}
\usepackage{subcaption}

\usepackage[utf8]{inputenc}
\allowdisplaybreaks

\DeclareCaptionSubType*[Alph]{figure}
\title{A hyperbolic model for two-layer thin film flow with a perfectly soluble anti-surfactant}

\author{Rahul Barthwal}
\author{Christian Rohde}
\affil{\footnotesize Institute of Applied Analysis and Numerical Simulation, University of Stuttgart,\\
Pffafenwaldring 57, Stuttgart, Germany}

\begin{document}
\maketitle
\begin{abstract}
We consider the motion of a two-layer thin film that consists of two immiscible viscous fluids and is endowed with an anti-surfactant solute. The presence of such solute particles induces variations of the surface tension and interfacial stress driving a Marangoni-type flow. We first analyze a lubrication limit and derive one-dimensional evolution equations for film heights and solute concentrations. Then, under the assumption that the capillarity and diffusion effects are negligible and the solute is perfectly soluble, we obtain a conservative first-order system in terms of film heights and concentration gradients. This reduced system is found to be strictly hyperbolic for a certain set of states and to admit an entire class of entropy/entropy-flux pairs. We also provide a strictly convex entropy for the hyperbolic system. Thus, the well-posedness for the Cauchy problem is given. Moreover, the system is almost a Temple-class system, which allows to compute explicit solutions of the Riemann problem. The paper concludes with numerical experiments using a Godunov-type finite volume method, which relies on the exact Riemann solver.
\end{abstract}
\medskip 
{\textbf{MSC codes.}  35L40 35L45 35L65 76A20  78M35

\section{Introduction}
The study of thin film models is a 
rich and intricate research area. This holds in particular true if the dynamics is subject to the influence of soluble or insoluble substances, inducing variations of the surface tension. These flows find important applications in technical and environmental fields like coating technologies, thin film solar cell technologies, surfactant replacement therapies, etc.  (see e.g.\ \cite{craster2009dynamics, Kutteretal, matar2004rupture, o2002theory}). The surface-tension driven flow is usually referred to as a Marangoni flow, which has been extensively studied by several researchers in the last two decades (see e.g.\ \cite{billingham2006surface, gaver1990dynamics, jensen1992insoluble,  keller1983surface,  matar2004rupture, myers1998thin}). The variations in surface tension of the fluid can occur due to temperature changes or by the introduction of surfactants or anti-surfactants. While surfactants decrease the surface tension of the fluid, there are other solutes which behave in an opposite manner. When the molecules of a dissolved solute are expelled from the free surface of a solvent, the surface tension of the solution increases. Such solutes are usually referred to as ``anti-surfactants". Many salts, such as sodium chloride or short-chain alcohols, when added to water can be considered to act as anti-surfactants.

The large-scale dynamics of a pure one-layer thin film dynamics is typically governed by a first-order equation for the film height. Accounting for gravity effects, one is led to a scalar conservation law with non-convex fluxes, which allow for so-called overshoot solutions containing undercompressive shock waves, see \cite{cook2008shock}. Such equations have been analyzed from the 
theoretical point of view in e.g.\  \cite{bertozzi1999undercompressive} and with regard to numerics in e.g.\ \cite{Chalonsetal}.

More recently, Conn et al.~\cite{conn2017simple, conn2016fluid} proposed and analyzed a mathematical model for a two-dimensional anti-surfactant solution in a one-layer thin film flow.  They obtained a reduced one-dimensional hyperbolic form out of it assuming that the solute is perfectly soluble while neglecting diffusion and capillarity effects. In one spatial dimension, it reads for unknown film thickness $h=h(x, t)$  and bulk concentration gradient $b=b(x, t)$ as
\begin{equation}
\label{eq: temple}
 \dfrac{\partial h}{\partial t}+\dfrac{\partial }{\partial x}\left(\dfrac{h^2b}{2}\right)=0,\qquad
    \dfrac{\partial b}{\partial t}+\dfrac{\partial}{\partial x}\left(\dfrac{hb^2}{2}\right)=0.
\end{equation} 
Notably, the system \eqref{eq: temple} belongs to the Keyfitz-Kranzer class type hyperbolic systems \cite{temple1983systems}, for which the shock and rarefaction curves coincide and form a straight line in the state space. The system \eqref{eq: temple} was studied for classical wave interactions by Minhajul et al.~in \cite{sekhar2019stability} and for nonclassical wave interactions by Sen and Raja Sekhar in \cite{sen2020delta}. Barthwal et al.~\cite{barthwal2023construction} were able to extend the aforementioned model to its two-dimensional counterpart and obtained self-similar and non-self-similar solutions for the Riemann problems \cite{barthwal2022two, barthwal2023construction}, which were further explored to solve a three-quadrant Riemann problem by Pandey et al.~\cite{anamika2024Riemann}. Barthwal et al. \cite{barthwal2026existence} recently developed the global existence of weak entropy solutions for a class of $2\times 2$ hyperbolic systems, including the system \eqref{eq: temple}, using a novel viscous approximation.

Most of the mathematical models related to thin film flows are restricted to either single-layer flows or to homogeneous fluids. To the best of the authors' knowledge, there are not many mathematical treatises of models related to two-layer thin film flows under the influence of a surfactant or anti-surfactant. However,  there is the work of Br\"ull \cite{bruell2016modeling} who established local well-posedness for two-layer thin film flow under the influence of an insoluble surfactant.  
The evolution equations accounting for the presence of anti-surfactants are different from the evolution equations for the case of surfactants-induced surface tension variations. In the case of insoluble surfactants, one does not require separate concentration evolution equations in each layer, while in the case of anti-surfactants, it is necessary to obtain the evolution equations of concentration in each layer.

In this paper, we consider such a thin film flow model accounting for an anti-surfactant solute. In the first part of the article, we formulate a one-dimensional lubrication approximation for the evolution of the two-layer thin film flow with a perfectly soluble anti-surfactant solute. For the full-dimensional two-layer flow, we start with two immiscible, incompressible, and viscous fluids lying on top of each other on a solid substrate. One representative example of such two-layer flows is a layer of water lying over an oil layer. Based on the full Navier-Stokes equations and the surfactant transport equations, we derive as first main result the evolution equations for film thickness and bulk concentration in each layer using the thin film approximation for a selected flow regime. This leads to a fourth-order system of governing equations that generalizes  previously found descriptions, e.g.~for one-layer configurations as in \cite{conn2016fluid}.
Moreover, by assuming that the capillarity and diffusion effects are negligible, we obtain as the second main result a novel reduced ($4\times 4$)-system of first-order conservation laws in terms of the dimensionless film thickness in the lower and the upper layer and the corresponding concentration gradients, see \eqref{eq: Main_system} in Section \ref{sec: 3}. We show that the reduced system is strictly hyperbolic in a subset of the state space with positive concentration gradients. We consider this property to be essential also for the well-posedness of the system with finite capillarity and diffusion numbers.  In the case of only one thin film layer, the system \eqref{eq: Main_system} collapses to system \eqref{eq: temple}.\\ 
The reduced system \eqref{eq: Main_system} has a remarkable structure. It almost belongs to the Temple class, for which exact Riemann solvers are known, and global existence of solutions has been achieved \cite{bressan2000convergence, MR892021}. 
Our novel system has properties quite similar to a Temple-class system, as it can be decoupled along its Riemann invariants and also has a semi-triangular structure. However, it still does not belong to the Temple class because the relations for shock and rarefaction waves in the second characteristic field differ from each other and are not represented by straight lines in the state space. This makes the solution theory more complicated than its one-layer counterpart.  Also, the family of Lax curves is more intricate in the two-layer case,  and requires the solution of nonlinear algebraic relations. In the one-layer system, the Lax curves can be connected with each other in an affine manner, which is not the case here.

In the second part of the article, we succeed to show local well-posedness for the Cauchy problem for the reduced system. Notably, we can identify a whole family of entropy/entropy-flux pairs for the reduced system. We also find a strictly convex entropy such that the reduced system belongs to the class of Friedrichs-symmetrizable systems. This becomes possible because we were able to identify a novel set of Riemann invariants, which forms a coordinate system for the hyperbolic system. A similar approach for a ($2\times 2$)-triangular system of conservation laws was utilized by Andreianov et al.~\cite{andreianov2015attainable}. Of course, two-layer thin film models cover more physics than the one-layer versions, but the rich mathematical structure of the reduced system is what makes this paper more interesting. Such hyperbolic systems with a class of entropies and a coordinate system are referred as a rich class of hyperbolic systems as pointed out by Serre \cite{MR1087091}. We refer the interested reader to the classical paper by Conlon and Liu \cite{conlon1981admissibility} for more details on Riemann invariants as a coordinate system. Finding such a new instance of a non-linear hyperbolic system with three or more equations is interesting from the point of view of the theory of hyperbolic systems. We note that hyperbolic systems with Riemann invariants as a coordinate system can admit global well-posedness even for weak solutions and the construction of accurate and high-order numerical solvers, see e.g. \cite{barthwal2025generalized, serre1999systems, MR4732267}. 

The findings on the Riemann invariants enable us to provide all rarefaction-wave curves and Rankine-Hugoniot loci. Thus, we can determine the exact solutions to the Riemann problem for the reduced hyperbolic system \eqref{eq: Main_system} by solving a set of nonlinear algebraic equations in the state space. In this way, we are automatically led to a finite volume scheme with Godunov flux. The scheme is validated, and some numerical simulations are shown. We compare our results with a Riemann solver free Lax-Friedrichs scheme, which has been established as a prototype stable and convergent scheme for many one-dimensional hyperbolic systems. Comparison with the Lax-Friedrichs scheme validates the accuracy and robustness of the constructed Riemann solver. 

The remainder of the article is organized as follows. Section \ref{sec:2} is devoted to developing the mathematical model for thin film flow and the evolution equations of film thickness and bulk concentration in each layer. In Section \ref{sec: 3}, we obtain the reduced first-order system, show that the reduced model is strictly hyperbolic,  and discuss its basic properties. Section \ref{sec: 4} is devoted to derive entropy/entropy-flux pairs with convex entropy and hence to verify the local well-posedness of the Cauchy problem for the reduced hyperbolic system. In Section \ref{sec: 5}, we discuss a Riemann problem for the reduced hyperbolic system and supplement it with some numerical experiments in Section \ref{sec: 6}. Concluding remarks are provided in Section \ref{sec: 7}.
\begin{figure}
    \centering
    \includegraphics[width=0.7\linewidth]{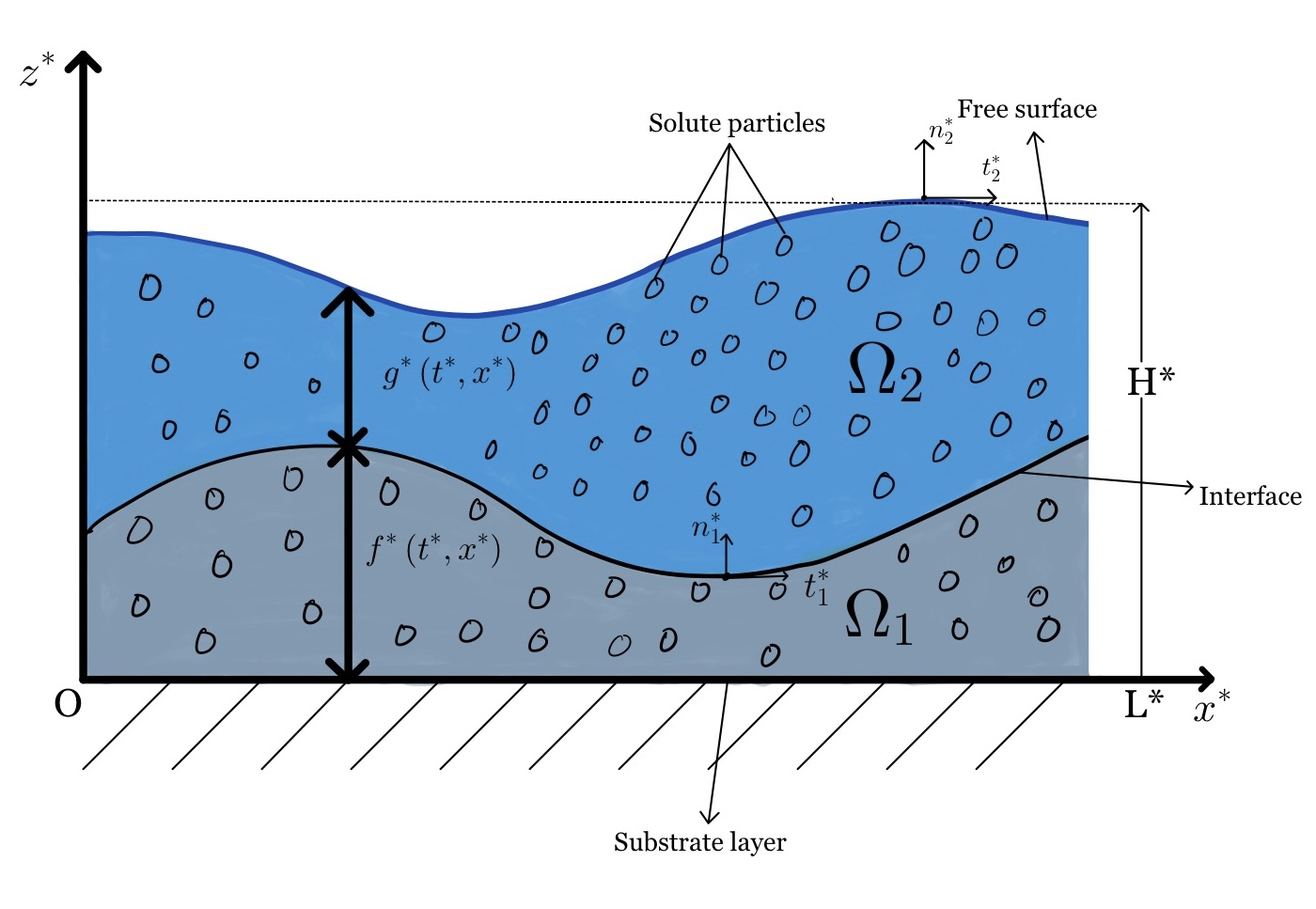}
    \caption{{Geometry of the two-layer thin film flow under the influence of an anti-surfactant.}}
    \label{fig:thin-film}
\end{figure}
\section{A two-layer  thin film flow model with perfectly soluble anti-surfactant\label{sec:2}}
\subsection{Governing equations and boundary conditions}\label{subsec: 2.1}
We consider the dynamics of two viscous, incompressible, Newtonian, and immiscible fluid films occupying the two-dimensional,
 time-dependent open sets
 $\Omega_1$ and  $\Omega_2$, see Figure \ref{fig:thin-film}.    The lower region $\Omega_1$ is bounded from below
by a horizontal impermeable and fixed bottom, whereas the upper region $\Omega_2$ is bounded from above by a free surface. The regions are separated by an impermeable fluid-fluid interfacial curve. We assume that both films contain a perfectly soluble solute, which increases the surface tension (anti-surfactant solute). Moreover, it is assumed that the interfacial stress separating the fluids is independent of external influences and that the external flow dynamics does not affect the film evolution. All quantities in this section are dimensional. \\ 
 Let $L^*$ be the length of the film and denote the characteristic film height by $H^*$.  For our study, we assume that the film heights are small as compared to the film lengths, i.e., for the ratio we have $H^*/L^*=\epsilon$ with $\epsilon\ll 1$. Thus, it is justified to neglect the effects of gravity.
 We denote by $x^*$ and $z^*$ the horizontal and vertical directions, respectively. Moreover, the two film heights are supposed to be given for time $t^*$ by the graphs of functions $f^*=f( x^*, t^*)$ and $g^*=g^* (x^*, t^*$) such that the interface is located at $z^*=f^*( x^*, t^*)$, and the free surface is located at $z^*=(f^*+g^*)( x^*, t^*)$. The bottom layer is located at $z^*=0$. 
The outward unit normal $ {\mathbf{\hat{n}_1}}^* $ and unit  tangent vector $ \mathbf{\hat{t}_1}^*$  of the interface are given for $ n_1 = \sqrt{1+(\partial_{x^*} f^*)^2}$ by 
\begin{align*}
{\mathbf{\hat{n}_1}}^*= n^{-1}_1 \left(-\partial_{x^*} f^*, 1\right)^\top
\text{ and }\, 
\mathbf{\hat{t}_1}^*=   n^{-1}_1   \left(1,\partial_{x^*} f^*\right)^\top.
\end{align*}
Analogously, with $n_2=\sqrt{1+\left(\partial_{x^*} (f^*+g^*)\right)^{2}}$ we define for the  free surface 
\begin{align*}
\mathbf{\hat{n}_2^*}=n_2^{-1}\left(-\partial_{x^*} (f^*+g^*),  1\right)^\top
\text{ and }\,
\mathbf{\hat{t}_2}^*=n_2^{-1}\left(1,\partial_{x^*} (f^*+g^*)\right)^\top.
\end{align*}
The dynamics of the two immiscible fluids is governed by the Navier-Stokes equations. The transport of bulk concentrations of the anti-surfactant adheres to advection-diffusion equations. As for the geometrical quantities, we mark the physical quantities in $\Omega_i$ by the subscript $i\in \{1, 2\}$. The bulk diffusion coefficient of the fluid in $\Omega_i$ is $D_{b_i}^*$, the density is $\rho_i^*$, and the viscosity is $\mu_i$, which are all assumed to be constants for each layer. The velocity vector is denoted by $\mathbf{u}_i^*=(u_i, w_i)^\top$, the pressure by $p_i^*$ and the bulk concentration by $c_i^*$. Then the corresponding equations of motion read for $i\in \{ 1, 2\}$ as \cite{picardo2015modelling}
\begin{align}\label{eq: 2.1main}
    \hspace{5.3 cm} \nabla^*. \mathbf{u}_i^*&=0, \nonumber\\
    \rho_i^*\left(\partial_{t^*} \mathbf{u}_i^*+\mathbf{u}_i^*. \nabla^* \mathbf{u}_i^*\right)+\nabla^* p_i^*-\mu_i^* \nabla^{*2} \mathbf{u}_i^*&=0,\vspace{0.15 cm}\nonumber\\
 \hspace{3.8 cm} \partial_{t^*} c_i^*+\mathbf{u}_i^*.\nabla^* c_i^*&=D_{b_i}^* \nabla^{*2} c_i^*,
\end{align}
where $\nabla^*=\left(\partial_{ x^*},  \partial_{ z^*}\right)$ denotes the gradient vector. 

The governing equations \eqref{eq: 2.1main} are to be supplemented with boundary conditions. Since the bottom surface of the substrate at $\{z^*=0\}$ is impermeable and stationary, we impose 
\begin{align}
    \mathbf{u_1^*}=\mathbf{0} ~~~\mathrm{on}~~\{z^*=\mathbf{0}\}.
\end{align}

We further assume that the velocity field is continuous across the  interface $\{z^*=f^*\}$, which implies 
\begin{align}
\mathbf{u_1^*}=\mathbf{u_2^*}~~\mathrm{at}~~ \{z^*=f^*\}.
\end{align} 

The normal stress balance equations at the interface $\{z^*=f^*\}$ and the free surface $\{z^*=f^*+g^*\}$ are
\begin{align}\label{eq: 9}
    \mathbf{[T_1^* -T_2^*]\hat{n}_1^*\cdot \hat{n}_1^*}&=\sigma_1^* \kappa_1^* ~~~\mathrm{on}~~\{z^*=f^*\},\\
    \hspace{1.1 cm}\mathbf{T_2^*\hat{n}_2^*\cdot \hat{n}_2^*}&=\sigma_2^* \kappa_2^* ~~~\mathrm{on}~~\{z^*=f^*+g^*\},\nonumber
\end{align}
where $\mathbf{T_i^*}=-p_i^* \mathbf{I}+\mu_i^*(\nabla^* \mathbf{u^*_i}+(\nabla^* \mathbf{u^*_i})^T)$ denotes the stress tensor in the $i${th} layer, $\sigma_1^*$ is the interfacial tension and $\sigma_2^*$ is the surface tension. We discuss the precise forms of $\sigma_1^*=\sigma^*_1(c_1^*)$ and $\sigma_2^*= \sigma_2^*(c^*_2)$  in \eqref{eq: 2.14main} and \eqref{eq: 2.15main} below. Further, \[
\kappa^*_1=\dfrac{\partial_{x^*}^2 f^*}{(1+(\partial_{x^*}^2 f^*)^2)^{\frac{3}{2}}}   \text{ and }\kappa_2^*=\dfrac{\partial_{x^*}^2 (f^*+g^*)}{(1+(\partial_{x^*}^2 (f^*+g^*))^2)^{\frac{3}{2}}}\] 
are the mean curvatures of the interface and the free surface, respectively.

Similarly, the tangential stress balance equations are given by
\begin{align}\label{eq: 10}
    \mathbf{[T_1^* -T_2^*]\hat{n}_1^*\cdot \hat{t}_1^*}&=   n^{-1}_1 \big(  {\partial_{x^*} \sigma_1^*+\partial_{z^*} \sigma_1^*\partial_{x^*} f^*} \big)
 ~ \qquad\qquad\mathrm{on}~~\{z^*=f^*\},\vspace{0.1 cm}\\
    \hspace{1.1 cm}\mathbf{T_2^*\hat{n}_2^*\cdot \hat{t}_2^*}&= n^{-1}_2 \big( {\partial_{x^*} \sigma_2^*+\partial_{z^*} \sigma_2^*\partial_{x^*} (f^*+g^*)}\big) 
    ~~~~\mathrm{on}~~\{z^*=f^*+g^*\}.\nonumber
\end{align}

Further, at the interface $\{z^*=f^*\}$ and the free surface $\{z^*=f^*+g^*\}$, kinematic boundary conditions are implemented, given by
\begin{align}
\partial_{t^*} f^*+u_1^*\partial_{x^*} f^*&=w_1^*~~~\mathrm{on}~~\{z^*=f^*\},\vspace{0.3 cm}\\
\partial_{t^*} g^*+u_2^*\partial_{x^*} g^*&=w_2^*~~~\mathrm{on}~~\{z^*=f^*+g^*\}.\nonumber
\end{align}

Since the bottom substrate is assumed to be impermeable, the bulk concentration in $\Omega_1$ must obey the no-flux condition given by
\begin{align}
\partial_{z^*} c_1^*=0~~~ \mathrm{at} ~\{z^*=0\}.
\end{align}
At the interface $\{z=f\}$, we impose the boundary conditions
\begin{align}\label{eq: interface}
     \nabla^* c_1^*.\mathbf{\hat{n}_1^*}&=\nabla^* c_2^*.\mathbf{\hat{n}_1^*}=0,
\end{align}
which implies that there is no solute exchange along the normal direction of the interface.
\begin{remark}\label{interfacial_remark}
The solute boundary conditions \eqref{eq: interface} at the interface are a consequence  of the choice which has been made in \cite{ kalogirou2019role, picardo2016solutal}, i.e.,
\begin{align*}
        D_{b_1}^* \nabla^* c_1^*.\mathbf{\hat{n}_1^*}&=D_{b_2}^* \nabla^* c_2^*.\mathbf{\hat{n}_1^*},
        \,\, c_1^*=K_{eq}c_2^*.\nonumber
\end{align*}
Here,  $K_{eq}$ is a distribution coefficient, which measures the distribution of solute particles in each layer when equilibrium is reached. Under the non-degeneracy condition $K_{eq}D_{b_1}-D_{b_2}\neq 0$ we get \eqref{eq: interface}.  For physical background, we refer the interested reader to  \cite{picardo2015modelling}. 
\end{remark}
Finally, since the solute is assumed to be perfectly soluble, the bulk concentration in $\Omega_2$ satisfies the total flux balance at the free surface $\{z^*=f^*+g^*\}$ given by 
\begin{align}\label{flux_free_surface}
-D_{b_2}^* \nabla^* c_2^*.\mathbf{\hat{n}_2^*}=0,
\end{align}

Now we are ready to define a constitutive relationship for  the interfacial tension $\sigma_1^*$ and the surface tension $\sigma_2^*$. Following again \cite{picardo2016solutal}, we choose a linear interfacial tension law at the interface. We assign to the interface a  thickness on each side 
denoted by $\eta_{11}^*$ and $\eta_{12}^*$, respectively such that the interfacial concentration  of  antisurfactants (negative adsorption) can be  approximated by
\[
C_{\rm{interface}}^*:=-\eta_{11}^* c_1^*-\eta_{12}^* c_2^*. 
\]
If we assume equilibrium (see Remark \ref{interfacial_remark}) we can  express $C_{\rm{interface}}^*$ in terms of $c^*_1$ only, i.e., $C_{\rm{interface}}^*=-\eta_{11}^* c_1^*- \eta_{12}K_{\rm eq}^{-1} c_1^*$. 
Using the Gibbs law and introducing the effective interfacial thickness $\eta_1^*= \eta_{11}^*+ \eta_{12}^*{K^{-1}_{eq}} $, the interfacial tension law takes the form  
\begin{align}\label{eq: 2.14main}
\sigma_1^*=\sigma_0^*+R^*T^* \eta_1^* c_1^*.
\end{align}
In \eqref{eq: 2.14main}, $R^*$ is the ideal gas constant, $T^*$ is the constant temperature, $\sigma^*_0$ is the interfacial tension of the pure solvent, and $\eta_1^* c_1^*$ is the effective interfacial concentration. 

For the free surface, we also use a  linear surface-tension law from \cite{barthwal2023construction,conn2016fluid} and get with analogous notations
\begin{align}\label{eq: 2.15main}
\sigma_2^*=\sigma_{s}^*+R^*T^*\eta_2^* c_2^*,
\end{align}
where $\sigma_s^*$ is the surface tension of pure solvent at the upper layer and $\eta_2^*$ is the assigned thickness of the free surface.
\begin{remark}
(i) To derive the surface tension relations \eqref{eq: 2.14main}  and \eqref{eq: 2.15main} we have implicitly used the assumption that the anti-surfactant is perfectly soluble in the interface and the surface. For more  general laws we refer to the literature  (cf.~eq.~(6) in  \cite{conn2016fluid}). These treat cases such that 
the ratio of the rate of particle adsorption to, and desorption from the interface or the free surface deviate from one.\\ 
(ii) Note that we employ simplified interfacial conditions in this article. In particular, we assume that the interfacial trace concentrations coincide with their corresponding bulk values, i.e., $c_{i, \rm{int}}^*=c_i^*$ for $i\in \{1, 2\}$, which is true in the leading order, see \eqref{c_expansion} below. Under this approximation and with  equilibrium boundary condition $c_1^*=K_{eq}c_2^*$, the interfacial stress law becomes a function of $c_1^*$ only. In a more realistic setting, the interfacial stress generally depends on the coupled interfacial composition of both layers. In particular, the trace values of interfacial concentrations need not to be equal to the bulk concentrations, and must be determined from interfacial equilibrium and mass-transfer constraints. Retaining these effects would introduce additional interfacial unknowns and a two-way coupling between the transport problems in the two fluids, substantially enriching and complicating the resulting dynamics. However, for the purpose of this article, we restrict attention to the simplified boundary conditions described above, which capture the leading-order influence of solute variations on interfacial stresses while keeping the model analytically and computationally tractable.
\end{remark}
\subsection{Non-dimensionalization}
We non-dimensionalize the system \eqref{eq: 2.1main} using scalings in terms of the ratio  $\epsilon$ of characteristic height $H^*$ and characteristic length $L^*$. Keeping in mind that the dimensional values are represented with stars, whereas non-dimensional quantities are represented without,  we consider  for $i\in\{1, 2\}$
\begin{align*}
&x=\dfrac{x^*}{L^*},~~z=\dfrac{z^*}{H^*},~~f=\dfrac{f^*}{H^*},~~g=\dfrac{g^*}{H^*},~~u_i=\dfrac{u_i^*}{U^*},~~w_i=\dfrac{w_i^*}{\epsilon U^*},~~t=\dfrac{L^*}{U^*}t^*,\\
&p_i=\dfrac{\epsilon^2 L^*}{\mu_i^* U^*}p_i^*,  ~~\sigma_1=\dfrac{\sigma_1^*}{\sigma_{0}^*},~~\sigma_2=\dfrac{\sigma_2^*}{\sigma_{s}^*},~~c_1=\dfrac{c_1^*}{C_1^*},~~c_2=\dfrac{c_2^*}{C_2^*}.
\end{align*}
Here, $U^*$ is the characteristic velocity scale, and $C_i^*, ~i\in \{1, 2\}$ is the typical bulk concentration scale at equilibrium in each layer. We introduce the non-dimensional density $\rho_i^*$ and viscosity $\mu^*_i$.
The re-scaling of physical quantities follows the scenario that has been outlined in \cite{barthwal2023construction, conn2017stability}: we choose 
\[
H^*= O(\epsilon^{-1}), \, L^* = O(\epsilon^{-2}),  U^*=O(\epsilon^3), C^*_i = O(\epsilon^{-2}), 
\] and 
$\rho_i^*,\, \mu^*_i = O(1)$. 
In terms of the non-dimensional variables, the governing equations in component form using the $(x, z)$-Cartesian coordinate system for $i \in \{1, 2\}$ become then
\begin{align}\label{eq: 8}
   ~~~~~~~~~\partial_x u_i+\partial_z w_i&=0,\nonumber\\
    \epsilon^2\left(\partial_t u_i+u_i\partial_x u_i+w_i\partial_z u_i\right)&=(\mathrm{Re_i})^{-1}\left(-\partial_x p_i+\epsilon^2 \partial_x^2 u_i+\partial_z^2 u_i\right),\vspace{0.2 cm} \nonumber\\
    \epsilon^4\left(\partial_t w_i+u_i\partial_x w_i+w_i\partial_z w_i\right)&=(\mathrm{Re_i})^{-1}\left(-\partial_z p_i+\epsilon^4 \partial_x^2 w_i+\epsilon^2\partial_z^2 w_i\right),\vspace{0.2 cm}\\
    \epsilon^2\left(\partial_t c_i+u_i\partial_x c_i+w_i\partial_z c_i\right)&=(\mathrm{P_{b_i}})^{-1}\left(\epsilon^2\partial_x^2 c_i+\partial_z^2 c_i\right).
\end{align}
In \eqref{eq: 8}, $\mathrm{Re_i}=(\rho_i^* U^* L^*)/\mu_i^*$ is the Reynolds number of the $i$th film and $\mathrm{P_{b_i}}=(U^*L^*)/D_{b_i}^*$ is the $i$th bulk P\'{e}clet number. For  $i\in \{1, 2\}$, we obtain with $D^*_{b_i} = O(\epsilon)$  the following scalings 
\[\mathrm{Re_i}=O(\epsilon), \quad\mathrm{P_{b_i}}=O(1).\]
For more physical details on these scalings, we refer the interested reader also to \cite{conn2017stability}.

Moreover, we consider the following asymptotic expansion for $c_i$ \cite{barthwal2023construction, conn2017stability}
\begin{align}\label{eq: asym}
c_i (x, z, t)=c_i^0(x, z, t)+\epsilon^2 c_i^1(x, z, t)+O(\epsilon^4).
\end{align}
Then by neglecting terms of order $\epsilon^2$, we obtain from \eqref{eq: 8} the reduced equations for $i \in \{1, 2\}$ given by 
\begin{subequations}\label{eq:reduced}
\begin{align}
    \partial_x u_i+\partial_z w_i&=0, \label{eq: 2.3a} \\
    \partial_x p_i&=\partial_z^2 u_i,\label{eq: 2.3b}\\
\partial_z p_i&=0, \label{eq: 2.3d}\\
    \partial_z^2 c_i^0&=0, \label{eq: 2.3e}\\
    \partial_t c_i^0+u_i\partial_x c_i^0&=(\mathrm{P_{b_i}})^{-1}\left(\partial_x^2 c_i^0+\partial_z^2 c_i^1\right). \label{eq: 2.3e_new} 
\end{align}
\end{subequations}
Note that \eqref{eq: 2.3e_new} still contains the second-order term  $c^1_i$ of the expansion of $c_i$ in \eqref{eq: asym}.
We proceed with the non-dimensionalization of the boundary conditions. In view of \eqref{eq: 2.3e}, we observe that $\partial_z c_i^0$ is constant for all $z$. In particular, due to the boundary condition at the substrate, we have
\begin{align}
  \partial_z c_1^0=0~\mathrm{at}~\{z=0\}. 
\end{align}
Due to the interface boundary condition $\eqref{eq: interface}$, and by neglecting second-order terms, we get $ \partial_z c_1^0=\partial_z c_2^0=0$ at $\{z=f\}$. Further the flux balance at free surface \eqref{flux_free_surface} implies that $ \partial_z c_2^0=0$ at $\{z=f+g\}$. Altogether, the bulk concentration $c_i$ is independent of $z$ to leading order $O(\epsilon^2)$, i.e., the leading order bulk concentration is uniform across the layer and is therefore given by some function $c_i=c_i^0(x,t)$. The asymptotic expansion of $c_i$ reduces to
\begin{align}\label{c_expansion}
c_i (x, z, t)=c_i^0(x, t)+\epsilon^2 c_i^1(x, z, t)+O(\epsilon^4).
\end{align}
Moreover, the substrate and flux boundary conditions for $c_i^1$ become
\begin{align}\label{eq: 2.7}
   \partial_z c_1^1\big|_{z=0}&=0,~~\qquad\qquad \partial_z c_1^1\big|_{z=f}\quad=\partial_x c_1^0~\partial_x f,\\
\partial_z c_2^1\big|_{z=f}&=\partial_x c_2^0~\partial_x f,~\quad\partial_z c_2^1\big|_{z=f+g}=\partial_x (f+g)~\partial_x c_2^0.\label{eq: 2.8}
\end{align}
Further, using \eqref{eq: 2.14main} and \eqref{eq: 2.15main}, the scaled interface tension coefficient after non-dimensionalization is $\sigma_1=1+\alpha_1 c_1^0$ and the surface tension coefficient is $\sigma_2=1+\alpha_2 c_2$. Scaling  $C^*_i = O(\epsilon^{-2})$  and $ \eta_1^*=O(\epsilon^4)$  we are led to  $\alpha_1=R^* T^* \eta_1^* C_1^*/\sigma_0^* =O(\epsilon^2) $ and $\alpha_2=R^* T^* \eta_2^* C_2^*/\sigma_s^*=O(\epsilon^2)$. Then, neglecting the higher-order term and simplifying, we obtain the  non-dimensional boundary conditions
\begin{subequations}\label{eq: 2.17main}
\begin{alignat}{2}
u_1 &= 0, \quad w_1=0,                                      &\quad \mathrm{on}~\{z=0\},       \label{eq: 2.4a}\\
u_1 &= u_2,    \quad w_1=w_2,                                  &\quad \mathrm{on}~\{z=f\},       \label{eq: 2.4b}\\
\partial_t f + u_1\,\partial_x f
    &= w_1,                                    &\quad \mathrm{on}~\{z=f\},       \label{eq: 2.4c}\\
\partial_t g + u_2\,\partial_x g
    &= w_2,                                    &\quad \mathrm{on}~\{z=f+g\},     \label{eq: 2.4d}\\
\mu\,p_2 - p_1
    &= (\mathrm{Ca}_1)^{-1}\partial_x^2 f,    &\quad \mathrm{on}~\{z=f\},       \label{eq: 2.4e}\\
-\,p_2
    &= (\mathrm{Ca}_2)^{-1}\partial_x^2(f+g), &\quad \mathrm{on}~\{z=f+g\},     \label{eq: 2.4f}\\
\partial_z u_1
    &= \mu\,\partial_z u_2 + \mathrm{Ma}_1\,\partial_x c_1^0,
                                              &\quad \mathrm{on}~\{z=f\},       \label{eq: 2.4g}\\
\partial_z u_2
    &= \,\mathrm{Ma}_2\bigl(\partial_xc_2
       + \partial_z c_2\,\partial_x(f+g)\bigr),
                                              &\quad \mathrm{on}~\{z=f+g\},     \label{eq: 2.4h}\\
\partial_z c_1^1
    &= 0,                                      &\quad \mathrm{on}~\{z=0\},       \label{eq: 2.4i}\\
\partial_z c_1^1
    &= \partial_x c_1^0\,\partial_x f,        &\quad \mathrm{on}~\{z=f\},       \label{eq: 2.4j}\\
\partial_z c_2^1
    &= \partial_x c_2^0\,\partial_x f,        &\quad \mathrm{on}~\{z=f\},       \label{eq: 2.4k1}\\
\partial_z c_2^1
    &= \partial_x(f+g)\,\partial_x c_2^0,          &\quad \mathrm{on}~\{z=f+g\}.     \label{eq: 2.4k2}
\end{alignat}
\end{subequations}

In \eqref{eq: 2.17main}, $\mu=\mu_2^*/\mu_1^*$ denotes the ratio of the viscosities of two fluids, $\mathrm{Ma_1}=(\epsilon R^* T^* \eta_1^* C_1^*)/\mu_1^* U^*$ and $\mathrm{Ma_2}=(\epsilon R^* T^* \eta_2^* C_2^*)/\mu_2^* U^*$ denote Marangoni numbers and $\mathrm{Ca_1}=(\mu_1^* U^*)/\epsilon^3 \sigma_{0}^*$, $\mathrm{Ca_2}=(\mu_2^* U^*)/\epsilon^3 \sigma_{s}^*$ are the capillarity numbers. Note that all these characteristic numbers are of order one with the scaling choices made before. 
\subsection{Evolution equation for the film heights}
Using the reduced equations of motion \eqref{eq:reduced} and the reduced boundary conditions \eqref{eq: 2.17main}, the velocity profiles $u_1$ and $u_2$ are given by
\begin{align}
   &u_1=\dfrac{\mu \partial_{x}^3(f+g)}{\mathrm{Ca_2}}z\left((f+g)-\dfrac{z}{2}\right)+\dfrac{\partial_{x}^3 f}{\mathrm{Ca_1}}z\left(f-\dfrac{z}{2}\right)+\mu z\mathrm{Ma_2}\partial_x c_2 +z\mathrm{Ma_1} \partial_x c_1^0,\label{eq: 2.18main} \\
&u_2=\dfrac{\partial_{x}^3(f+g)}{\mathrm{Ca_2}}\bigg[f\left(g+\dfrac{f}{2}\right)(\mu-1)+z\left((f+g)-\dfrac{z}{2}\right)\bigg]+\dfrac{\partial_{x}^3 f}{\mathrm{Ca_1}}\dfrac{f^2}{2}\label{eq: 2.6a_velocity_2}\\
&\hspace{4 cm}+\mathrm{Ma_1}  \partial_x c_1^0 f+\mathrm{Ma_2}\partial_x c_2\left(f(\mu-1)+z\right).\nonumber
\end{align}
The detailed calculations can be found in Appendix \ref{sec: appendix}. The evolution equations for the film thickness can then be obtained using the kinematic boundary conditions \eqref{eq: 2.4c}-\eqref{eq: 2.4d} and the continuity equations from \eqref{eq: 8}, and write as
\begin{align}\label{eq: 2.24main}
   \partial_t f&+\partial_x \left(\dfrac{f^2(\mathrm{Ma_1} \partial_x c_1^0+\mu \mathrm{Ma_2}\partial_x c_2^0)}{2}\right)=\partial_x Q_1,\\
    \partial_t g&+\partial_x\left(fg\mathrm{Ma_1} \partial_x c_1^0+\mathrm{Ma_2}\partial_x c_2^0\left(g\left(\mu f+\dfrac{g}{2}\right)\right)\right)=\partial_x Q_2,\label{eq: 2.24second}
\end{align}
with the third-order derivative terms  $Q_1$ and $Q_2$ 
given by
\begin{align}\label{eq: 2.22main}
    &Q_1=-\dfrac{\mu \partial_{x}^3(f+g)}{\mathrm{Ca_2}}\dfrac{f^2}{2}\left(\dfrac{2f}{3}+g\right)-\bigg[\dfrac{\partial_{x}^3 f}{\mathrm{Ca_1}}\bigg]\dfrac{f^3}{3}\\
&Q_2=-\dfrac{\partial_{x}^3(f+g)}{\mathrm{Ca_2}}\bigg[\left(\dfrac{(\mu-1)(f^2 g+2fg^2)}{2}\right)+\dfrac{(f+g)^3}{3}-\dfrac{f^2(2f+3g)}{6}\bigg]-\dfrac{f^2 g~ \partial_{x}^3 f}{2~\mathrm{Ca_1}}.\label{eq: 2.23main}
\end{align}
We note that both terms $Q_1$ and $Q_2$ are inversely proportional to the respective capillarity numbers  $\mathrm{Ca_1}$   and $\mathrm{Ca_2}$.
\subsection{Evolution equations for the bulk concentrations}
To obtain the  evolution equation for 
the bulk concentration $c_1^0$, we integrate \eqref{eq: 2.3e_new} with respect to $z$ from $\{z=0\}$ to $\{z=f\}$ with boundary conditions \eqref{eq: 2.4i} and \eqref{eq: 2.4j}, which gives
\begin{align}\label{eq: 2.34main}
  f\partial_t c_1^0+\dfrac{f^2 \partial_x c_1^0}{2}\left(\mathrm{Ma_1} \partial_x c_1^0+\mu \mathrm{Ma_2}\partial_x c_2^0\right)=(\mathrm{P_{b_1}})^{-1}\left(f\partial_x^2 c_1^0+\partial_x c_1^0~\partial_x f\right)+Q_1 \partial_x c_1^0,\nonumber\\
 f\partial_t c_1^0+\dfrac{f^2 \partial_x c_1^0}{2}\left(\mathrm{Ma_1} \partial_x c_1^0+\mu \mathrm{Ma_2}\partial_x c_2^0\right)=(\mathrm{P_{b_1}})^{-1}\left(\partial_x(f\partial_x c_1^0)\right)+ Q_1\partial_x c_1^0,
\end{align}
where $Q_1(x, t)$ is defined in \eqref{eq: 2.22main}.

In a similar fashion, we can find the evolution equation of $c_2^0$ by integrating \eqref{eq: 2.3e_new} from $\{z=f\}$ to $\{z=f+g\}$ and using the boundary conditions \eqref{eq: 2.4k1} and \eqref{eq: 2.4k2}. The evolution equation for $c_2^0$ is given by
\begin{align}\label{eq: 2.35main}
  g\partial_t c_2^0&+\left(\mathrm{Ma_1} \partial_x c_1^0fg+\mathrm{Ma_2}\partial_x c_2^0\left(g(\mu f+\dfrac{g}{2})\right)\right) \partial_x c_2^0=(\mathrm{P_{b_2}})^{-1}\partial_x\left(g\partial_x c_2^0\right)+Q_2\partial_x c_2^0,
\end{align}
where $Q_2$ is defined in \eqref{eq: 2.23main}.
\subsection{A fourth-order  two-layer thin film model}\label{subsec:fourth}
To summarize the discussion above, we collect all the obtained evolution equations  for $f,g,c^0_1,c^0_2$. We obtain as a first main result of our work  the closed fourth-order nonlinear system
\begin{equation}\label{eq:main}
\begin{array}{rcl}
  \partial_t f
  +\partial_x \left(\dfrac{f^2(\mathrm{Ma_1} \partial_x c_1+\mu \mathrm{Ma_2}\partial_x c_2)}{2}\right)
  &=&\partial_x Q_1,
  \\[1.9ex]
\partial_t g
  +\partial_x\left(
      fg\,\mathrm{Ma_1} \partial_x c_1
      +\mathrm{Ma_2}\partial_x c_2\left(g\left(\mu f+\dfrac{g}{2}\right)\right)
    \right)
  &=&\partial_x Q_2,
  \\[1.5ex]
  f\partial_t c_1
  +\left(\dfrac{f^2}{2}
      \left(\mathrm{Ma_1} \partial_x c_1+\mu \mathrm{Ma_2}\partial_x c_2\right)\right)\partial_x c_1 
      
  & =& \, (\mathrm{P_{b_1}})^{-1}\partial_x(f\partial_x c_1)
      +Q_1 \partial_x c_1,
      \\[1.9ex]
g\partial_t c_2
  +\left(
      \mathrm{Ma_1} \partial_x c_1fg
      +\mathrm{Ma_2}\partial_x c_2\left(g\left(\mu f+\dfrac{g}{2}\right)\right)
    \right)\partial_x c_2
 
  &
   =&\,(\mathrm{P_{b_2}})^{-1}\partial_x\left(g\partial_x c_2\right)
      +Q_2 \partial_x c_2. 
\end{array}
\end{equation}
Note that we have dropped the superscripts in the quantities $c_1^0,c^0_2$, which denote the leading order bulk concentrations.

\begin{remark}
We recover the evolution equations for height and concentration for the case of one-layer thin film flow obtained by Conn et al.~in \cite{conn2017stability} from the evolution equations \eqref{eq:main} by taking $f=0, ~c_1=0$  (see eq. (4.7.1) \&(4.7.3) in \cite{conn2017stability}).
For the sake of simplicity, we derived the evolution equations for two-dimensional sets $\Omega_1,\, \Omega_2$. Using the same computations, one can handle three-dimensional domains, see \cite{barthwal2023construction} for one-layer systems.
\end{remark}

\section{A two-layer thin film model on the large scale}\label{sec: 3}
We identify in this section a first-order sub-system of the evolution equations 
\eqref{eq:main}. It governs the large-scale
dynamics and refers to a specific flow regime,  both of which we explain in detail 
in Section \ref{subsec:reduction}.  We show in Section \ref{subsec: 3.2} that the sub-system is a strictly hyperbolic system of conservation laws\cite{dafermos2005hyperbolic}.

\subsection{The first-order system for the asymptotic dynamics of the thin film model}\label{subsec:reduction}
The equations \eqref{eq:main} contain the 
third-order terms $Q_i$ that vanish if we set for the inverse of the capillarity numbers  $ (\mathrm{Ca_i})^{-1}=0$.  Furthermore we neglect diffusion and choose $ (\mathrm{Pb_i})^{-1}=0$. 
The remaining equations are given by 
\begin{equation}\label{eq:maind}
\begin{array}{rcl}
  \partial_t f
  +\partial_x \left(\dfrac{f^2(\mathrm{Ma_1} \partial_x c_1+\mu \mathrm{Ma_2}\partial_x c_2)}{2}\right)
  &=&
  {0,}
  \\[1.9ex]
  \partial_t g
  +\partial_x\left(
      fg\,\mathrm{Ma_1} \partial_x c_1
      +\mathrm{Ma_2}\partial_x c_2\left(g\left(\mu f+\dfrac{g}{2}\right)\right)
    \right)
  &=&{0,}
  \\[1.9ex]
 f\partial_t c_1
  +\left(\dfrac{f^2}{2}
      \left(\mathrm{Ma_1} \partial_x c_1+\mu \mathrm{Ma_2}\partial_x c_2\right)\right)\partial_x c_1 &=&0,
 \\[1.9ex]
  g\partial_t c_2
  +\left(
      \mathrm{Ma_1} \partial_x c_1fg
      +\mathrm{Ma_2}\partial_x c_2\left(g\left(\mu f+\dfrac{g}{2}\right)\right)
    \right)\partial_x c_2
  & =&0.
\end{array}
\end{equation}
The assumptions about the capillarity 
and  P\'{e}clet numbers select a specific flow regime. But regardless of whether the assumptions are justified, the equations \eqref{eq:maind} describe the 
basic dynamics of our two-layer system. 
A particularly interesting  sub-system is found when we set  $\mu =0$ and  choose the notional thicknesses  $\eta_i^*$ and the concentration reference numbers  $C_i^*$ in such a way that $\mathrm{Ma_1}= \mathrm{Ma_2}=1$
holds.  The reduced system of  evolution equations for the large-scale dynamics is then given by
\begin{equation}
\begin{array}{rcl}
\label{eq: 2.9}
\dfrac{\partial f}{\partial t}+\dfrac{1}{2}\dfrac{\partial}{\partial x}\left(f^2\dfrac{\partial c_1}{\partial x}\right)&=&0,\vspace{0.2 cm}\\[0.4em]
\dfrac{\partial c_1}{\partial t}+\dfrac{1}{2}f\left(\dfrac{\partial c_1}{\partial x}\right)^2&=&0,\\[1.9ex]
\dfrac{\partial g}{\partial t}+\dfrac{\partial}{\partial x}\left(\dfrac{1}{2}g^2\dfrac{\partial c_2}{\partial x}+fg\dfrac{\partial c_1}{\partial x}\right)&=&0,\\[0.4em]
\dfrac{\partial c_2}{\partial t}+\dfrac{1}{2}g\left(\dfrac{\partial c_2}{\partial x}\right)^2+f\left(\dfrac{\partial c_1}{\partial x}\right)\left(\dfrac{\partial c_2}{\partial x}\right)&=&0.
\end{array}
\end{equation}
To uncover the conservation law structure of \eqref{eq: 2.9} we differentiate the second and the fourth equations in $\eqref{eq: 2.9}$ with respect to $x$, respectively, and  obtain
for the film thicknesses $f$, $g$ and  the concentration gradients 
\begin{align*}
b=\dfrac{\partial c_1}{\partial x}, \quad 
q=\dfrac{\partial c_2}{\partial x}
\end{align*}
the following system of  first-order conservation laws: 
\begin{align}\label{eq: Main_system}
    \dfrac{\partial f}{\partial t}+\dfrac{\partial}{\partial x}\left(\dfrac{1}{2}f^2b\right)&=0,\vspace{0.2 cm} \nonumber\\
    \dfrac{\partial b}{\partial t}+\dfrac{\partial}{\partial x}\left(\dfrac{1}{2}fb^2\right)&=0,\vspace{0.2 cm} \\
     \dfrac{\partial g}{\partial t}+\dfrac{\partial}{\partial x}\left(\dfrac{g^2q}{2}+fgb\right)&=0,\vspace{0.2 cm} \nonumber\\
  \dfrac{\partial q}{\partial t}+\dfrac{\partial}{\partial x}\left(\dfrac{gq^2}{2}+fbq\right)&=0\nonumber
\end{align}

The remainder of the paper is devoted to a detailed study of the  system \eqref{eq: Main_system} which lays the foundations to understand \eqref{eq:maind} and even the fourth-order system from Section \ref{subsec:fourth}.

\begin{remark}
\begin{itemize}
\item[(i)] The system \eqref{eq:maind} could have been also directly derived by an appropriate scaling of the  capillarity 
and  P\'{e}clet numbers in terms of the height-to-length ratio $\varepsilon$. But then we would have not found the fourth-order system \eqref{eq:main} which holds for any choice of the numbers. In what follows we provide a rather complete theory for the first-order system \eqref{eq: Main_system}.    
%
%
\item[(ii)] The assumption that  the viscosity ratio $\mu$ is very small holds for e.g.\ oil-water type flows. Clearly, this means that the bottom layer should be more viscous than the upper layer. From the mathematical  point of view, we observe that  the choice $\mu =0$ decouples  the first two equations in \eqref{eq: Main_system}. Also, we recover the one-layer system \eqref{eq: temple} by setting $f=b=0$.
%
%
%
\item[(iii)]
Note that the equilibrium condition  $c_1=K_{eq}c_2$ at the interface can be recovered from the reduced first-order system. For $K_{eq}>0$, if we have  $K_{eq}f-g=0$, then at the level sets $fb=gq$ we  obtain the interfacial boundary conditions \eqref{eq: interface}  because $f, b, g, q$ satisfy $g/f=b/q=K_{eq}$. In Section \ref{sec: Riemann_invariants} and Section \ref{sec: 5}, we observe that these ratios can be expressed as Riemann invariants of the system \eqref{eq: Main_system}. Note, however, that we 
do not consider the case  $fb=gq$, which may give rise to a possible resonance phenomenon in a purely first-order theory.
\end{itemize}
\end{remark}

The quasilinear system \eqref{eq: Main_system} can be written in the compact conservative form 
\begin{align}
&\mathbf{U}_t+(\mathbf{F}{(\mathbf{U}))}_x=0,\label{conservation-laws}
\end{align}
where $\mathbf{U}=(f,b,g,q)^{\top}\in \mathcal{U}$ denotes the vector of conservative variables and 
\[\mathbf{F}(\mathbf{U})=\left(\frac 12 f^2b,\frac 12fb^2,\frac 12 g^2q+fbg,\frac 12 gq^2+fbq\right)^{\top}\]
is the flux vector. The state space $\mathcal{U}$ is introduced as
\begin{align}\label{statespace}
    \mathcal{U}= \Big\{(f, b, g, q)^\top\in \mathbb{R}^4: f, b, g, q>0, fb<gq \Big\}.
\end{align}
For the well-posedness of the full system \eqref{eq: 2.24main}-\eqref{eq: 2.35main} it is essential that the reduced first-order system \eqref{eq: Main_system} is hyperbolic and equipped with a proper entropy structure
in the state space. We study these issues in the 
remainder of the article.

\begin{remark}
The state space $\mathcal{U}$ is restricted to the case $f>0, b>0, g>0, q>0$ and $fb<gq$ which ensures the strict hyperbolicity of the system \eqref{eq: Main_system} as we will see in Section \ref{subsec: 3.2}, see also Remark 
\ref{Rem_hyp}.
\end{remark}
\subsection{Strict hyperbolicity of the reduced system \eqref{eq: Main_system}}\label{subsec: 3.2}
For smooth solutions, we can convert the system $\eqref{eq: Main_system}$ into its primitive form as
\begin{align*}
&\mathbf{U}_t+\mathbf{DF}(\mathbf{U})  \mathbf{U}_x=0,
\end{align*}
with the flux Jacobian 
\begin{align}
&\mathbf{DF}(\mathbf{U})=\left(
\begin{array}{cccc}
fb & \frac 12 f^2 & 0 & 0\\
\frac 12 b^2 & fb & 0 & 0\\
gb & fg & fb+gq & \frac 12 g^2\\
bq & fq & \frac 12 q^2 & fb+qg \\
\end{array}
\right). \nonumber
\end{align}
Due to its block matrix structure, the eigenvalues $\lambda_k= \lambda_k(\mathbf{U})$ of $\mathbf{DF}(\mathbf{U})$ can be explicitly  computed for 
$k=1,\ldots,4$. They are real and obey  for $\mathbf{U}=(f,b,g,q)^{\top}\in \mathcal{U}$ the ordering
\begin{align}\label{eigenvalues}
\lambda_1(\mathbf{U})=\dfrac{fb}{2}<~\lambda_2(\mathbf{U})=\dfrac{3fb}{2}<~\lambda_3(\mathbf{U})=fb+\dfrac{gq}{2}<~\lambda_4(\mathbf{U})=fb+\dfrac{3gq}{2}.
\end{align}
Moreover, the (right) eigenvectors $\mathbf{r_k}= \mathbf{r_k}(\mathbf{U})$ are
\begin{equation}\label{righteigenvectors}
\begin{array}{rcl}\displaystyle 
\mathbf{r_1}(\mathbf{U})= \left(-\frac{f}{b},1,0,0\right)^{\top}, && \!\!\!\!\!\! \displaystyle \mathbf{r_2}(\mathbf{U})=\left(\frac{fb-3gq}{4qb},\frac{fb-3gq}{4qf},\frac{g}{q},1\right)^{\top},\\[3ex]\displaystyle 
\mathbf{r_3}(\mathbf{U})=\left(0,0,-\frac{g}{q},1\right)^{\top}, & & \!\!\!\!\!\! \displaystyle \mathbf{r_4}(\mathbf{U})=\left(0,0,\frac{g}{q},1\right)^{\top}.
\end{array}
\end{equation}
The eigenvectors are linearly independent in $\mathcal U$.  This implies the strict hyperbolicity of the reduced system \eqref{eq: Main_system} in $\mathcal U$.

Next, we obtain with the eigenvector formulas for the four characteristic fields \begin{equation}\label{GL}\nabla_{\mathbf{U}} \lambda_1\cdot\mathbf{r_1}=\nabla_{\mathbf{U}} \lambda_3\cdot\mathbf{r_3}=0 \text{ and } \nabla_{\mathbf{U}} \lambda_2\cdot  \mathbf{r_2}=3fb \text{ and }\nabla_{\mathbf{U}} \lambda_4\cdot  \mathbf{r_4}=3gq. 
\end{equation}
Thus, the first and the third characteristic fields are linearly degenerate, while the second and the fourth characteristic fields of \eqref{eq: Main_system} are genuinely nonlinear in $\mathcal{U}$ \cite{dafermos2005hyperbolic}. Note that $\nabla_{\mathbf{U}}$ denotes the gradient of any field with respect to the vector $\mathbf{U}=(f, b, g, q)^T$. The properties in \eqref{GL} will be essential for solving the Riemann problem for \eqref{eq: Main_system} in Section 
\ref{sec: 5}.
\begin{remark} (i)\label{Rem_hyp}
We consider here only the case $f,b,g,q >0$  for which strict hyperbolicty can be shown. The same holds for negative concentration gradients and for concentration gradients with  $b<0$  and $q>0$. 
We conjecture that a similar analysis to the following one can be done in these cases. The case $q<0$ is not clear since then a Riemann invariant gets complex, see \eqref{Rinvariants} below. For a more complete study of different choices of state spaces, we refer to \cite{barthwal2025generalized}. 
\\
(ii)
In order to maintain strict hyperbolicity, we assume that $gq>fb$ holds for $\mathbf{U} \in \mathcal U$, see \eqref{statespace}. 
For the opposite case $fb>gq$, we have also strict hyperbolicity, but the ordering of eigenvalues changes, and the second and third wave change their role as long as $fb<3gq$. In the critical case $fb=gq$, eigenvalues coincide. The system remains hyperbolic and possesses a full set of eigenvectors. Due to the presence of the full set of eigenvectors, one can still define the rarefaction curves as the integral curves of the second and fourth characteristic fields. However, additional resonant waves may occur due to the recurrence of eigenvalues; see Remark \ref{rem_resonance}. 
\end{remark}
\section{Entropy/entropy-flux pairs and  well-posedness}\label{sec: 4}
Hyperbolicity is a necessary condition to ensure the well-posedness of the initial-value problem for \eqref{eq: Main_system}.  In this section, we show that the reduced system has much more structure.
Precisely, we propose a class of entropy/entropy-flux pairs for \eqref{eq: Main_system} and identify a specific pair with convex entropy. By an entropy/entropy-flux pair, we mean a pair 
$(E, Q): \mathcal{U}\mapsto \mathbb{R}\times \mathbb{R}$ such that $(E, Q)$ satisfies 
\begin{equation}\label{entropy}
\nabla_{\mathbf{U}} E(\mathbf{U})^\top\mathbf{DF(U)}=\nabla_{\mathbf{U}}Q(\mathbf{U})^\top
 \text{ for all $\mathbf{U}\in \mathcal U$}.
\end{equation}
The compatibility condition \eqref{entropy} implies in particular that any smooth solution $\mathbf{U}=\mathbf{U}(x,t)$ of \eqref{eq: Main_system} is entropy conservative, i.e., 
\[{E(\mathbf{U}(x,t))}_t+{Q(\mathbf{U}(x,t))}_x=0 \mbox{ in $\mathbb{R}\times(0,\infty)$}
\] holds. For more implications in case of weak solutions, we refer to \cite{dafermos2005hyperbolic,serre1999systems}.
\subsection{Riemann invariants}\label{sec: Riemann_invariants}
To explore the entropic structure of \eqref{eq: Main_system}, we rely on the finding that the strictly hyperbolic system \eqref{eq: Main_system} is equipped with a full set of  Riemann invariants. Indeed, three $k$-Riemann invariants $\Gamma_1^k,\Gamma_2^k,\Gamma_3^k$ corresponding to the $k\rm{th}$-characteristic field, $k=1,\ldots, 4$, are given by 
\begin{align}
\begin{cases}
\Gamma_1^1=g,~\quad\Gamma_2^1=q,~\quad\Gamma_3^1=fb,\\
    \Gamma_1^2 ={b}/{f},~\Gamma_2^2={q}/{g},~\Gamma_3^2={(fb+gq)}/{(gq)^{1/4}},\\
    \Gamma_1^3=f,~\quad\Gamma_2^3=b,~\quad\Gamma_3^3=gq,\\
   \Gamma_1^4=f,~\quad\Gamma_2^4=b,~\quad\Gamma_3^4={q}/{g}.
\end{cases}   \label{Rinvariants} 
\end{align}
From \eqref{Rinvariants} we select the four Riemann invariants 
\begin{equation}\label{selectedinvariants}
\xi=w_1=b/f, \quad     u= w_2 = fb,  \quad \tau = w_3 = q/g, \quad 
\eta = w_4=(fb+gq)/(gq)^{1/4}.
\end{equation}
These form a coordinate system for \eqref{eq: Main_system}, i.e., the mapping 
\[
\mathbf{U}  = (f,b,g,q)^\top \mapsto \mathbf{W}= (\xi,u,\tau,\eta)^\top
\]
is one-to-one in $\mathcal U$.  This is proven in the following Lemma.
\begin{lemma}\label{lemma}
The functions  $w_1, w_2, w_3, w_4$  from  \eqref{selectedinvariants}  form a coordinate system for system \eqref{eq: Main_system}. Furthermore,  the system \eqref{eq: Main_system} in terms of $\mathbf{W}= (w_1,w_2,w_3,w_4)^\top$ is equivalent to the diagonal system
\begin{align}\label{diagonal_form}w_{k,t}+\lambda_k(\mathbf{U}(\mathbf{W})){ w_{k,x} }=0 \qquad (k=1, \ldots,4).
\end{align}
\end{lemma}
\begin{proof}
A necessary and sufficient condition for $w_1, w_2, w_3, w_4$ to form a coordinate system  for \eqref{eq: Main_system} is that the gradients 
$\nabla_{\mathbf{U}}w_k$ 
are left eigenvectors of  the flux Jacobian  $\mathbf{D} \mathbf{F}$ (see Theorem 7.3.1 in \cite{dafermos2005hyperbolic}), i.e.,
\begin{align}\label{eq: Rinvariantsconditions}
\nabla_{\mathbf{U}}w_k(\mathbf{U}).\mathbf{r}_j(\mathbf{U})\begin{cases}
    \neq 0 ~{\rm{if}}~j=k,\\
    = 0 ~{\rm{if}}~j\neq k,
\end{cases}~\forall~k, j =1, \ldots, 4\qquad (\mathbf{U} \in \mathcal U),
\end{align}
with $\mathbf{r}_j$  being  the $j$th eigenvector of the reduced system \eqref{eq: Main_system}.

Here, we only prove the statement  \eqref{eq: Rinvariantsconditions} for the most complicated function  $w_4$ and leave the other similar computations to the reader. Now, we directly compute
\[\nabla_{\mathbf{U}}w_4(\mathbf{U})=\left(\dfrac{b}{(gq)^{1/4}}, \dfrac{f}{(bq)^{1/4}}, \dfrac{-q(fb-3gq)}{4(gq)^{5/4}}, \dfrac{-g(fb-3gq)}{4(gq)^{5/4}}\right)^\top.\]
It is straightforward to compute  
$\nabla_{\mathbf{U}}w_4(\mathbf{U})\cdot\mathbf{r_j}(\mathbf{U})=0$ for $j\in \{1, 2, 3\}$ while $\nabla_{\mathbf{U}}w_4(\mathbf{U})\cdot \mathbf{r_4}(\mathbf{U})=\dfrac{-g(fb-3gq)}{2(gq)^{5/4}}\neq 0$ for $\mathbf{U}\in \mathcal{U}$, which implies that $w_4$ satisfies \eqref{eq: Rinvariantsconditions}. 

For the equivalence of \eqref{eq: Main_system} and \eqref{diagonal_form}, we refer again to Chapter 7.3 in \cite{dafermos2005hyperbolic}.
\end{proof}
In terms of the variables $\xi,u,\tau,\eta$, the diagonal system \eqref{diagonal_form}  writes out for $v=gq$ as 
\begin{equation}\label{Rtransformed}
\begin{array}{rclrcl}
   \xi_t+\dfrac{u}{2} \xi_x&=&0,&  u_t+\dfrac{3 u}{2} u_x&=&0, \\[1.5ex]
   \tau_t+\left(u+\dfrac{v}{2}\right)\tau_x&=&0, &\quad 
\eta_t+\left(u+\dfrac{3v}{2}\right)\eta_x&=&0.
\end{array}    
\end{equation}
Before we return to the entropy discussion,  let us mention that $k$-Riemann invariants remain constant across $k$-rarefaction waves, which will be used in Section \ref{sec: 5} \cite{dafermos2005hyperbolic}. 
\begin{remark}
The statements of Lemma \ref{lemma} are an important finding of this paper. It is remarkable that there are explicit Riemann invariants for a system of hyperbolic conservation laws having three or more equations.   
\end{remark}
\subsection{Entropy/entropy-flux pairs for the system \eqref{eq: Main_system}\label{subsec:4.2}}
To prove \eqref{entropy} for some given pair $(E,Q)$, it suffices to verify the corresponding relation for the transformed system \eqref{Rtransformed}.
\begin{theorem}[Entropy/entropy-flux pairs] \label{theo1}Let $\rho, \mu$ and $\nu$ be arbitrary functions in  $C^1(0,\infty)$.
For the system \eqref{eq: Main_system}, there exists a class of entropy/entropy-flux pairs $( E, Q)$ which is 
given by
\begin{equation}\label{eq:entropyclass}
\begin{array}{rcl}
 E[f, b, g, q]&=& \rho(fb)+\sqrt{fb}\mu\left(\dfrac{b}{\sqrt{fb}}\right)+\sqrt{gq}\nu\left(\dfrac{q}{\sqrt{gq}}\right)+\dfrac{1}{fb+gq},\\[3ex]
 Q[f, b, g, q]&=&\psi(fb)+(fb)^{3/2} \mu\left(\dfrac{b}{\sqrt{fb}}\right)+(\sqrt{gq}) \nu\left(\dfrac{q}{\sqrt{gq}}\right)(fb+gq)\\
&&\hspace{3.5 cm}-\ln \left((fb+gq)^{3/2}\right)-\dfrac{fb}{2(fb+gq)}.
\end{array}
\end{equation}
The function $\psi$ is determined from $\psi'(w)=\dfrac{3}{2}  w\rho'(w)$.
\end{theorem}
\begin{proof} We have to prove the compatibility relation \eqref{entropy} for $(E,Q)$ from \eqref{eq:entropyclass}.  In view of  the bijectivity of the map $(\xi, u, \tau, \eta)\mapsto (f, b, g, q)$, it suffices to  show that 
the transformed compatibility relation holds for the functions $E[\xi, u, \tau, \eta] := E[f, b, g, q]$ and
$Q[\xi, u, \tau, \eta] := Q[f, b, g, q]$.  From \eqref{Rtransformed} we observe that the componentwise 
relations are given by 
\begin{equation}\label{eq: 4.6Relations}
Q_{\xi}=\dfrac{u}{2} E_{\xi},  \,  \,
    Q_u=\dfrac{3u}{2}E_u,   \,\, 
    Q_{\tau}=\left(u+\dfrac{v}{2}\right) E_{\tau},    \, \, 
    Q_{\eta}= \left(u+\dfrac{3v}{2}\right) E_{\eta}.
\end{equation} 
The quantity $v= gq$ is related to the functions $u$ and $\eta$ by the nonlinear relation $u+v=\eta v^{1/4}$.

The functions $E$ and $Q$  from \eqref{eq:entropyclass}  in terms of the Riemann invariants $\xi, u, \tau, \eta$   are given by 
\[\begin{array}{rcl}
E[\xi, u, \tau, \eta]&=& \rho(u)+\sqrt{u}\mu(\xi)+\sqrt{v}\nu(\tau)+\dfrac{1}{\eta v^{1/4}},\\[2ex]
Q[\xi, u, \tau, \eta]&= &\psi(u)+\dfrac{u^{3/2}}{2}\mu(\xi)+\dfrac{\sqrt{v}\nu(\tau)}{2}\left( \eta v^{1/4}+u\right)-\dfrac{3}{2}\ln\left(\eta v^{1/4}\right)-\dfrac{u}{2\eta v^{1/4}}.
\end{array}
\]
We readily  check that the pair $\left(E[\xi, u, \tau, \eta], Q[\xi, u, \tau, \eta]\right)$ satisfies \eqref{eq: 4.6Relations}.
 \end{proof}
We have provided a class of entropy/entropy flux pairs for system \eqref{eq: Main_system}. We can actually choose a generic strictly convex entropy within this class. For systems equipped with a set of Riemann invariants as a coordinate system, the convexity condition of an entropy $E$ gets simpler:  it is sufficient to check that the Hessian  $H_{ E}$ of $E$ satisfies the conditions (see Chapter 7.4 of \cite{dafermos2005hyperbolic})
\begin{equation} \label{reducedcriterion}\mathbf{r}_k(\mathbf{U})^\top H_{ E}(\mathbf{U})\, \mathbf{r_k(\mathbf{U})}>0 \qquad \text{for all }~k\in \{1, \ldots, 4\} \text{ and all $\mathbf{U}\in \mathcal U$.}
\end{equation}
\begin{theorem}[Convex entropy] \label{Entropytheorem}
The system \eqref{eq: Main_system} is endowed with at least one entropy/entropy-flux pair $(\bar E, \bar Q)$ with strictly convex entropy $\bar E$ in $\mathcal U$. It is given by
\begin{align}\label{convex entropy}
     \bar E[f, b, g, q]= \dfrac{1}{fb}+\dfrac{f^{3/2}}{\sqrt{b}}+\dfrac{1}{fb+gq}+\dfrac{g^{3/2}}{\sqrt{q}}.
\end{align}
The associated entropy flux computes as 
\begin{align*}
    \bar Q[f, b, g, q]= -\dfrac{3}{2}\ln\left((fb(fb+gq))\right)+ \dfrac{f^{5/2}\sqrt{b}}{2}-\dfrac{fb}{2(fb+gq)}+g^{5/2}\sqrt{q}\left(fb+\dfrac{gq}{2}\right).
\end{align*}    
\end{theorem}
\begin{proof}
With the choice $\rho(u) = u^{-1}, \, \mu(\xi) = \xi^{-1}$ and $ \nu(\tau) = \tau^{-1}$, the function $\bar{E}$ defined in \eqref{convex entropy}  belongs to the class of entropies for system \eqref{eq: Main_system} found in Theorem
\ref{theo1}. Moreover, by using the relation $\psi'(u)=\dfrac{3u\rho'(u)}{2}$, and using the values of $\mu$ and $\nu$ in the expression of $Q$ in Theorem \ref{theo1}, it is straightforward to obtain the expression for the entropy flux $\bar{Q}$. Thus, the pair $(\bar{E}, \bar{Q})$ forms an entropy/entropy-flux pair for the system \eqref{eq: Main_system}. Therefore, we just need to check the strict convexity of $\bar E$ via \eqref{reducedcriterion}. 
We directly compute the Hessian  $H_{\bar E}(\mathbf{U})$ of $\bar E(\mathbf{U})$, which reads as
\begin{align*}
    \left(\begin{array}{cccc}\frac{2b^2}{(bf+gq)^3}+\frac{8+3f^{2}\sqrt{fb}}{4bf^3}&\frac{4-3\sqrt{fb}}{4b^2f^2}+\frac{bf-gq}{(bf+gq)^3}&\frac{2bq}{(bf+gq)^3}&\frac{2bg}{(bf+gq)^3}\\
    \frac{4-3\sqrt{fb}}{4b^2f^2}+\frac{bf-gq}{(bf+gq)^3}&\frac{2f^2}{(bf+gq)^3}+\frac{8+3f^{2}\sqrt{fb}}{4b^3f}&\frac{2fq}{(bf+gq)^3}&\frac{2fg}{(bf+gq)^3}\\\frac{2bq}{(bf+gq)^3}&\frac{2fq}{(bf+gq)^3}&\frac{2q^2}{(bf+gq)^3}+\frac{3}{4\sqrt{gq}}&\frac{gq-bf}{(bf+gq)^3}-\frac{3\sqrt{g}}{4q^{3/2}}\\\frac{2bg}{(bf+gq)^3}&\frac{2fg}{(bf+gq)^3}&\frac{gq-bf}{(bf+gq)^3}-\frac{3\sqrt{g}}{4q^{3/2}}&\frac{2g^2}{(bf+gq)^3}+\frac{3g^{3/2}}{4q^{5/2}}\\\end{array}\right).
\end{align*}
A tedious but straightforward calculation leads to (skipping the $\mathbf{U}$-argument) the following positive numbers 
in $\mathcal U$:
\begin{align*}
\mathbf{r_1^\top } H_{\bar E} \mathbf{r_1}&=\frac{3f^{3/2}}{\sqrt{b}}+\frac{2bf}{(bf+gq)^2}+\frac{2}{bf},\quad
\mathbf{r_2^\top }H_{\bar E} \mathbf{r_2}=
\frac{6\left(3g^2q^2+2bf(gq-fb)\right)}{b^3f^3(3gq-fb)(bf+gq)},\\
\mathbf{r_3^\top } H_{\bar E} \mathbf{r_3}&=\frac{2gq}{(bf+gq)^2}+\frac{3g^{3/2}}{\sqrt{q}},~~\quad \qquad\mathbf{r_4^\top } H_{\bar E} \mathbf{r_4}=\frac{2gq(3gq-bf)}{(bf+gq)^3}.
\end{align*}
Thus, the Hessian of $\bar E$ satisfies \eqref{reducedcriterion} and the chosen entropy $\bar E$ is strictly convex.
\end{proof}
\begin{remark}\label{rem42}
The choice of the strictly convex entropy $\bar E$  is not unique, and other strictly convex entropies can be constructed based on the diagonal structure of the system \eqref{diagonal_form} and a proof similar to that of Theorem \ref{Entropytheorem}. 
In fluid mechanics, the entropy typically has a thermodynamical interpretation. We are not aware whether this is also true for $\bar E$ or other possible convex entropies. However, the expression of $\bar E$ indicates that for very small thin films, $\bar E$ increases rapidly, so to keep the system stable, film thickness should not be too low or the concentration gradient should be of a larger scale. For other possible state spaces, the existence of convex entropy can still be ensured due to the work of Conlon and Liu \cite{conlon1981admissibility}. Moreover, the entropy/entropy-flux pairs for the system \eqref{eq: Main_system} become the entropy/entropy flux pairs of the one-layer system \eqref{eq: temple} for $f=0$ and all values of $b$.
\end{remark}
\subsection{Local well-posedness of the Cauchy problem for the system \eqref{eq: Main_system}}
For general quasilinear systems of conservation laws, the Cauchy problem may not be well-posed. However, for Friedrichs-symmetrizable systems, the (local)  well-posedness of classical solutions is guaranteed (see \cite{dafermos2005hyperbolic, MR390516, majda2012compressible, serre1999systems} and references cited therein). If a quasilinear system of conservation laws is equipped with an entropy/entropy-flux pair with strictly convex entropy, it is a Friedrichs-symmetrizable system. In Section \ref{subsec:4.2} we found such a pair for the hyperbolic system \eqref{eq: Main_system}. Therefore, we obtain immediately the local well-posedness of the Cauchy problem for system \eqref{eq: Main_system} with initial data in the Sobolev space $(H^m(\mathbb{R}))^4$ for $m>3/2$ such that for a finite time $T<\infty$, there exists a unique solution $\mathbf{U}\in C([0, T]; (H^m(\mathbb{R}))^4)\cap C^1([0, T]; (H^{m-1}(\mathbb{R}))^4)$. 
\section{The Riemann problem}\label{sec: 5}
In this section, we first compute the elementary wave curves for the system \eqref{eq: Main_system}. Based on the wave curves and the entropy/entropy-flux pairs from Section  \ref{sec: 4}, we can then provide a complete entropy solution for the Riemann problem. This Riemann solver allows to construct a Godunov-type finite volume scheme for the system \eqref{eq: Main_system}. To solve the Riemann problem of the ($4\times4$)-system  \eqref{eq: Main_system}, it is crucial that the $4$th characteristic field
turns out to be a Temple field, i.e., the $4$th  shock-wave curve coincides with the integral curve that defines the rarefaction waves of the field \cite{temple1983systems}. 
\subsection{Elementary waves\label{subsec: 4.1}}
We determine the relations that are valid across the different types of admissible elementary waves (corresponding to $k$-shock waves, $k$-contact discontinuities and $k$-rarefaction waves) as algebraic equations. Admissibility refers to the property of an elementary wave to be an entropy solution of \eqref{eq: Main_system}. In particular, we find the Riemann invariants, which are constant across rarefaction waves, and we analyze the Rankine-Hugoniot conditions, which must hold across a discontinuous wave to render it a weak solution of \eqref{eq: Main_system}.
\subsubsection{Rarefaction waves}\label{subsubsec: rarefaction}
Before we construct rarefaction wave curves, we
recall that the Riemann invariants  $\Gamma^k_1,\Gamma^k_2,\Gamma^k_3$ from \eqref{Rinvariants} remain constant across a corresponding $k$-rarefaction wave \cite{dafermos2005hyperbolic}.
A $k$-rarefaction wave can only exist for a genuinely nonlinear field. Therefore we consider the cases $k\in \{2,4\}$, see \eqref{GL}.

We start with $k=2$. The slope inside a $2$-rarefaction  wave is given by
\begin{align*}
    \dfrac{dx}{dt}=\dfrac{x}{t}=\lambda_2=\dfrac{3fb}{2}.
\end{align*}
The characteristic speed increases across the rarefaction wave, which means that $fb\geq f_lb_l$ must hold for any left state $\mathbf{U_l}=(f_l, b_l, g_l, q_l)^{\top}\in \mathcal{U}$ across the $2$-rarefaction wave. 
Thus, using the Riemann invariants $\Gamma^2_1,\Gamma^2_2,\Gamma^2_3$ from \eqref{Rinvariants}, the solution inside the $2$-rarefaction wave is given for the left state $\mathbf{U}_l$ by 
\begin{align}\label{eq: 4.2R} R_2:=
    \begin{cases}
        \dfrac{dx}{dt}=\dfrac{x}{t}=\lambda_2=\dfrac{3fb}{2},\vspace{0.2 cm}\\
        \dfrac{f}{b}=\dfrac{f_l}{b_l},~\dfrac{fb+gq}{(gq)^{1/4}}=\dfrac{f_lb_l+g_lq_l}{(g_lq_l)^{1/4}}, ~\dfrac{g}{q}=\dfrac{g_l}{q_l},\vspace{0.2 cm}\\
        fb\geq f_lb_l.
    \end{cases}
\end{align}
$R_2$ is a curve in  the $(f, b, g, q)$-space emanating from $\mathbf{U}_l$.

Similarly, the $4$-rarefaction wave curve  from a left state $\mathbf{U}_l$ is given  by
\begin{align} \label{R4} R_4:=
    \begin{cases}
        \dfrac{dx}{dt}=\dfrac{x}{t}=\lambda_4=fb+\dfrac{3gq}{2},\vspace{0.2 cm}\\
        \dfrac{g}{q}=\dfrac{g_l}{q_l},~f=f_l, b=b_l,\vspace{0.2 cm}\\
        g_lq_l\leq gq.
    \end{cases}
\end{align}
\begin{remark}\label{rem_resonance}
Note that the existence of rarefaction waves can be ensured even for the possible resonant case $fb=gq$, which is excluded for states in  $\mathcal U$. The system \eqref{eq: Main_system} possesses a full set of eigenvectors even for $fb=gq$; see \eqref{righteigenvectors}. For the second characteristic field with a given left state $\mathbf{U}_l$ and $fb=gq$, the 2-rarefaction curve is obtained by solving the IVP
\[
\dfrac{d\mathbf{U}}{d \nu}=\mathbf{r}_2(\mathbf{U}), \quad \mathbf{U}(\nu_l)=\mathbf{U}_l, 
\]
where \[
\mathbf{r}_2(\mathbf{U})=\left(\dfrac{-g}{2b}, \dfrac{-g}{2f}, \dfrac{g}{q}, 1\right)^\top.
\]
In view of the smoothness of $\mathbf{r_2}(\mathbf{U})$, the existence of a unique integral curve (2-rarefaction) curve is an immediate consequence of Picard's existence theorem.
Note that this curve is different from the 3-contact discontinuity curve; see \eqref{eq: 23b} below. However, due to $fb=gq$, a resonance phenomenon is possible, which can cause composite wave structures due to the wave interaction of $2$-rarefaction and $3$-contact discontinuity waves. In particular, the unique solvability of corresponding Riemann problems may not be guaranteed anymore by the Lax entropy condition \cite{hu1998riemann, goatin_resonance}. An additional physically-based selection criterium  is needed, see also Remark \ref{Rem_hyp}(ii). 
\end{remark}
\subsubsection{Discontinuous waves: shock and contact waves}
For speed $\sigma \in \mathbb{R}$,  a discontinuous wave 
\begin{align} \label{wave}
    \mathbf{U} (x, t)=\begin{cases}
        \mathbf{U}_l=(f_l, b_l, g_l, q_l )^\top\in {\mathcal U}&:~x-\sigma t<0,\\
        \mathbf{U}_r=(f_r, b_r, g_r, q_r)^\top \in {\mathcal U}&:~x-\sigma t> 0,
    \end{cases}
\end{align}
is a distributional solution of \eqref{eq: Main_system} if the  Rankine-Hugoniot conditions 
\begin{align}\label{RH}
\sigma[\![\mathbf{U}]\!]  =[\![\mathbf{F(U)}]\!]
\end{align}
hold. Here,  $[\![\mathbf{W}]\!]=\mathbf{W}_r-\mathbf{W}_l$ denotes the jump in $\mathbf{W}$ for $\mathbf{W}_l, \mathbf{W}_r\in \mathbb{R}^4$.

For the system \eqref{eq: Main_system},  the componentwise Rankine-Hugoniot relations  are
\begin{equation}\label{eq: 4.8}
    \begin{array}{rclrcl}
    \sigma[\![f]\!]\!\!\!\!&=&\!\!\!\!\dfrac{1}{2}[\![f^2b]\!],&
    \sigma[\![b]\!]\!\!\!\!&=&\!\!\!\!\dfrac{1}{2}[\![fb^2]\!],\vspace{0.2 cm}\\
    \sigma[\![g]\!]\!\!\!\!&=&\!\!\!\!\Big[\!\!\Big[\dfrac{g^2q}{2}+fgb\Big]\!\!\Big],&
    \sigma[\![q]\!]\!\!\!\!&=& \!\!\!\!\Big[\!\!\Big[\dfrac{gq^2}{2}+fbq\Big]\!\!\Big].
    \end{array}
\end{equation}
Along the first and third linearly degenerate characteristic fields, for a given left state $\mathbf{U_l}=(f_l, b_l, g_l, q_l)^{\top}\in \mathcal{U}$, the discontinuity curves are  curves through $\mathbf{U_l}$ in the $(f, b, g, q)$-space, which are given by
\begin{align}\label{eq: 23a}
    J_1:=\begin{cases}
    \sigma_1=\dfrac{fb}{2}=\dfrac{f_lb_l}{2},\\ 
    fb=f_lb_l,~g=g_l, ~q=q_l
    \end{cases}
\end{align}
and 
\begin{align}\label{eq: 23b}
    J_3:=\begin{cases}
    \sigma_3=fb+\dfrac{gq}{2}=f_lb_l+\dfrac{g_lq_l}{2},\\ 
    f=f_l, ~b=b_l,  ~gq=g_lq_l.
    \end{cases}
\end{align}
The corresponding $1$- and $3$-contact waves \eqref{wave} are entropy solutions of \eqref{eq: Main_system} by definition.

For $k\in \{2,4\}$, the $k$th characteristic field is genuinely nonlinear. We are interested in left/right states that satisfy the  Rankine-Hugoniot condition \eqref{RH} and are entropy admissible. The function \eqref{wave} is then called $k$-shock wave.

For $k=2$ we get a three-parameter family of curves starting from the left state $\mathbf{U_l}$ by eliminating $\sigma_2$ from \eqref{eq: 4.8}. Given the strict hyperbolicity of the system \eqref{eq: Main_system}, the entropy admissibility of elements of $2-$shock wave ($S_2$) with respect to the entropy/entropy-flux pair \eqref{convex entropy} can be checked by the validity of the Lax entropy conditions. For $S_2$ with right state $\mathbf{U}$, these are
\begin{align}\label{eq: 25}
    \lambda_2(\mathbf{U})<\sigma_2<\lambda_2(\mathbf{U_l}),\,
 \lambda_1(\mathbf{U_l})  < \sigma_2.
\end{align}
From the first two equations of \eqref{eq: 4.8}, we first obtain the shock speed as 
\[\sigma_2=\dfrac{b_l(f_l^2+f_lf+f^2)}{2f_l}.\]
Therefore, by eliminating $\sigma_2$ from  \eqref{eq: 4.8} and using the entropy conditions \eqref{eq: 25}, we obtain the following nonlinear relations across $S_2$
\[\dfrac{f}{b}=\dfrac{f_l}{b_l}, \quad \dfrac{g}{q}=\dfrac{g_l}{q_l}, ~g=g_l+\Phi(f_l, b_l, g_l, q_l, f, b, g, q),\]
with $\Phi$ being the  nonlinear function
\begin{align*}
   \Phi(f_l, b_l, g_l, q_l, f, b, g, q):= \dfrac{2f_l}{b_l(f_l^2+f_lf+f^2)}\bigg(\dfrac{g^2 q}{2}-\dfrac{g_l^2q_l}{2}+fb g-f_lb_lg_l\bigg).
\end{align*}
In the one-layer system \eqref{eq: temple}, the shock curve becomes a simple curve in the solution space. However, for system \eqref{eq: Main_system}, this is not the case. To summarize the discussion above, the 2-shock curve is given by 
\begin{align}\label{S_2def}
   S_2:=\begin{cases} \sigma_2= \dfrac{b_l(f_l^2+f_lf+f^2)}{2f_l},\\ 
   \dfrac{f}{b}=\dfrac{f_l}{b_l}, ~\dfrac{g}{q}=\dfrac{g_l}{q_l},~g=g_l+\Phi(f_l, b_l, g_l, q_l, f, b, g, q),
   \\
   fb<f_lb_l,\end{cases}
\end{align}
With the same arguments, we identify $4$-shock waves which satisfy
\begin{align}\label{eq: 24b}
   S_4:=\begin{cases} \sigma_4= fb+\dfrac{q_l(g_l^2+g_lg+g^2)}{2g_l},\\ 
   \dfrac{q}{g}=\dfrac{q_l}{g_l},~f=f_l,~b=b_l,\\
   gq<g_lq_l.
   \end{cases}
\end{align}
We observe that the $4$-shock wave relations \eqref{eq: 24b} coincide with the  $4$-rarefaction relations in \eqref{R4} and form a straight line in the $(f, b, g, q)$-space. This implies that the $4$th characteristic field is a Temple field. However, note that the complete system does not belong to the  Temple class.
\subsection{Solution of the Riemann problem}\label{sec: Riemann}
Consider the  Riemann problem for \eqref{eq: Main_system}, that is, the Cauchy problem  with initial data of the form 
\begin{align}
    \mathbf{U} (x, 0)=\begin{cases}
        \mathbf{U}_L=(f_L, b_L, g_L, q_L )^\top\in {\mathcal U}&:~x<0,\\
        \mathbf{U}_R=(f_R, b_R, g_R, q_R)^\top \in {\mathcal U}&:~x>0.
    \end{cases}
\end{align}
Note that Lax curves overlap with each other in the one-layer thin film hyperbolic model, which simplifies the construction of the Riemann solver, which is not the case here. The relations across $R_2$ and $S_2$ render the Lax curves to be nonlinear and different from each other. Thus, it is not easy to connect the Lax curves of two different families for solving the Riemann problem. We provide a novel construction of the Riemann solver by proving that these nonlinearities can be converted into a set of nonlinear algebraic equations, for which we can find a solution.   

Due to the fact that the first and the third characteristic fields allow for contact discontinuities only, there are four wave configurations possible for the solution to the Riemann problem based on different choices of initial data.  In (\ref{RP-pic}-\ref{fig: case1(b)}), we display typical wave configurations with three intermediate states $\mathbf{U_L^*},\mathbf{U_M^*}, \mathbf{U_R^*} \in \mathcal U$. We list all four cases  as follows:
\begin{enumerate}
\item {\textbf{Case 1.}} If the initial data satisfies $f_Rb_R\geq f_Lb_L$ then the second wave is a $2$-rarefaction wave and there are only two possible solution structures:
\begin{itemize} 
    \item {\textbf{Case 1(a).}}  $J_1+R_2+J_3+S_4$.
    \item {\textbf{Case 1(b).}}  $J_1+R_2+J_3+R_4$.
\end{itemize}
\item {\textbf{Case 2.}} If the initial data satisfies $f_Rb_R< f_Lb_L$, the second wave will be a $2$-shock wave. We list the possible solution structure for this case as follows:
\begin{itemize}
    \item {\textbf{Case 2(a).}} $J_1+S_2+J_3+S_4$.
    \item {\textbf{Case 2(b).}} $J_1+S_2+J_3+R_4$.
\end{itemize}
\end{enumerate}
\subsubsection{Case 1}\label{subsubsec: 5.1.1}
Let us consider the elementary wave curves from Section \ref{subsec: 4.1} for the proposed configuration $J_1+R_2+J_3+S_4$ (see \ref{RP-pic}) or $J_1+R_2+J_3+R_4$ (see \ref{Case 1(b)}).

For the first and third characteristic fields, we have contact waves, and therefore we deduce from \eqref{eq: 23a}, \eqref{eq: 23b} the relations 
\begin{equation}\label{eq: 5.2J}
 g^*_L=g_L,\ q^*_L=q_L,\ f_Lb_L=f^*_Lb^*_L,\
 f^*_R=f_M^*,\ b^*_R=b_M^*,\  g_M^*q_M^*=g^*_Rq^*_R.
\end{equation}
The Riemann invariants for a $2$-rarefaction wave $R_2$ provide
\begin{align}\label{eq: 5.2R}
& \ \dfrac{f_L^*}{b_L^*}=\dfrac{f_M^*}{b_M^*},~\ \dfrac{g_L^*}{q_L^*}=\dfrac{g_M^*}{q_M^*}, ~\ \dfrac{f_L^* b_L^* +g_L^* q_L^*}{\left(g_L^* q_L^*\right)^{\frac 14}}=\dfrac{f_M^* b_M^* +g_M^* q_M^*}{\left(g_M^* q_M^*\right)^{\frac 14}}.
\end{align}
Since the $4$th characteristic field is a Temple field, the shock curve and the rarefaction curve coincide, and we have from \eqref{eq: 24b} and \eqref{R4}
\begin{align}\label{eq: 5.2S}
f_R^*=f_R,~b_R^*=b_R,~g_R^*/q_R^*=g_R/q_R.
\end{align}
Combining \eqref{eq: 5.2J}-\eqref{eq: 5.2S}, we observe that the intermediate states 
$\mathbf{U_L^*}, f^*_M, b^*_M, q^*_M, \mathbf{U_R^*} $ can be expressed in terms of the Riemann initial data and $g^*_M$ by the set of equations 
\begin{equation}\label{case 1 relations}
\begin{array}{c}
(f_L^*,~b_L^*,~g_L^*,~q_L^*)=\left(\sqrt{\dfrac{f_Lb_Lf_R}{b_R}}, \sqrt{\dfrac{f_Lb_Lb_R}{f_R}}, g_L, q_L\right),\\
    f_M^* = f_R,  \quad  b_M^* = b_R, \quad  q_M^*=g_M^* \dfrac{q_L}{g_L}, \\
    (f_R^*,~b_R^*,~g_R^*,~q_R^*)=\left(f_R,~ b_R, ~g_M^*\sqrt{\dfrac{q_Lg_R}{g_Lq_R}},~g_M^* \sqrt{\dfrac{q_Lq_R}{g_Lg_R}}\right).
\end{array}
\end{equation}
Using the last relation for the $2$-rarefaction wave $R_2$ from \eqref{eq: 5.2R}, the state $g_M^*$ is  obtained by solving  in both cases the equation  $F_1(g_M^*) =0$ with $F_1$ given by 
\begin{align}\label{eq: 5.2}
  F_1(g_M^*):= (g^*_M)^2 q_L-\sqrt{g_M^* g_L}(f_L b_L+g_Lq_L)+f_R b_Rg_L.
\end{align}
In view of the initial data condition $f_Rb_R\geq f_Lb_L$ and $f_Rb_R<g_Lq_L<3g_Lq_L$ for $\mathbf{U_{L/R}}\in \mathcal{U}$, we have $F_1(g_L)>0$ and $F
_1(g_{\rm{min}})<0$ for $g_{\rm{min}}=\left(\dfrac{\sqrt{g_L}(f_Lb_L+g_Lq_L)}{4q_L}\right)^{2/3}$. 
Therefore, a positive zero $g_M^*$ of $F_1$ exists in  either the interval $[g_L, g_{\rm{min}})$ or  in $[g_{\rm{min}}, g_L]$.

It is still not clear whether the fourth wave always corresponds to either a $4$-shock wave satisfying the entropy condition 
\begin{align}\label{4Srelations}
g_M^*q_M^*=g_R^*q_R^*>g_Rq_R,
\end{align}
or a  $4$-rarefaction wave, supposed to satisfy
\begin{align}
g_M^*q_M^*=g_R^*q_R^*\leq g_Rq_R.
\end{align}\
In any case, we have trivially either $g_M^*>\sqrt{\dfrac{g_Lg_Rq_R}{q_L}}  \text{ or } g_M^*\leq \sqrt{\dfrac{g_Lg_Rq_R}{q_L}}.$

Consider the case when $g_M^*>\sqrt{(g_Lg_Rq_R)/q_L}$. Squaring both sides and using  $q_M^*=(g_M^* q_L)/g_L$ from $\text{\eqref{case 1 relations}}_2$, we have $g_M^*q_M^*> g_Rq_R$. Clearly, given \eqref{4Srelations}, this implies that the fourth wave in this case is a $4$-shock wave. Analogously, we prove that the fourth wave is a $4$-rarefaction wave when  $g_M^*\leq \sqrt{(g_Lg_Rq_R)/q_L}$.

We conclude that the Riemann problem for Case 1 can be solved by either the solution structure 
$J_1+R_2+J_3+S_4$ or by $J_1+R_2+J_3+R_4$. 
\begin{figure}
\captionsetup[subfigure]{font=scriptsize,labelfont=scriptsize}
\centering
\begin{subfigure}{0.5\textwidth}
\centering
  \hspace{-1 cm}\includegraphics[width=2.4 in]{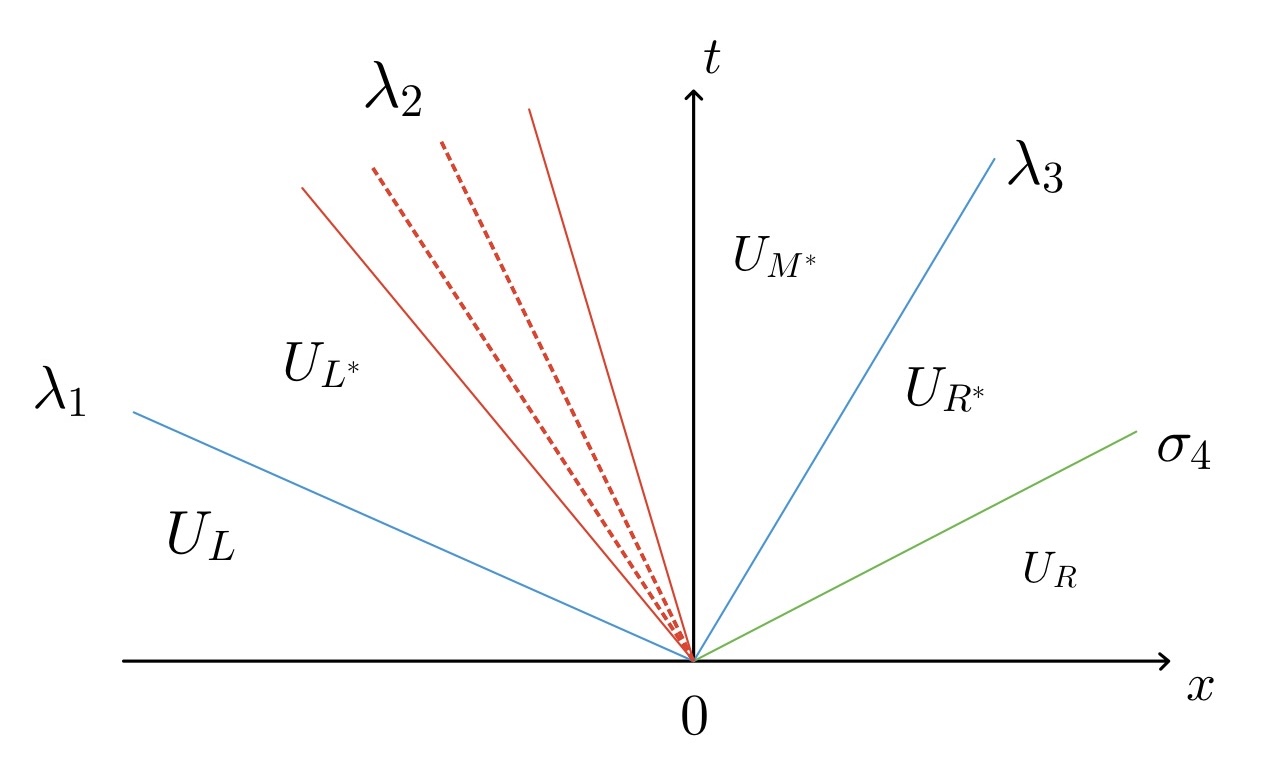}
  \caption{Typical wave configuration for\\
  Case 1(a). $J_1+R_2+J_3+S_4$.}\label{RP-pic}
\end{subfigure}%
\begin{subfigure}{0.5\textwidth}
\centering
  \hspace{-1 cm}\includegraphics[width=2.4 in]{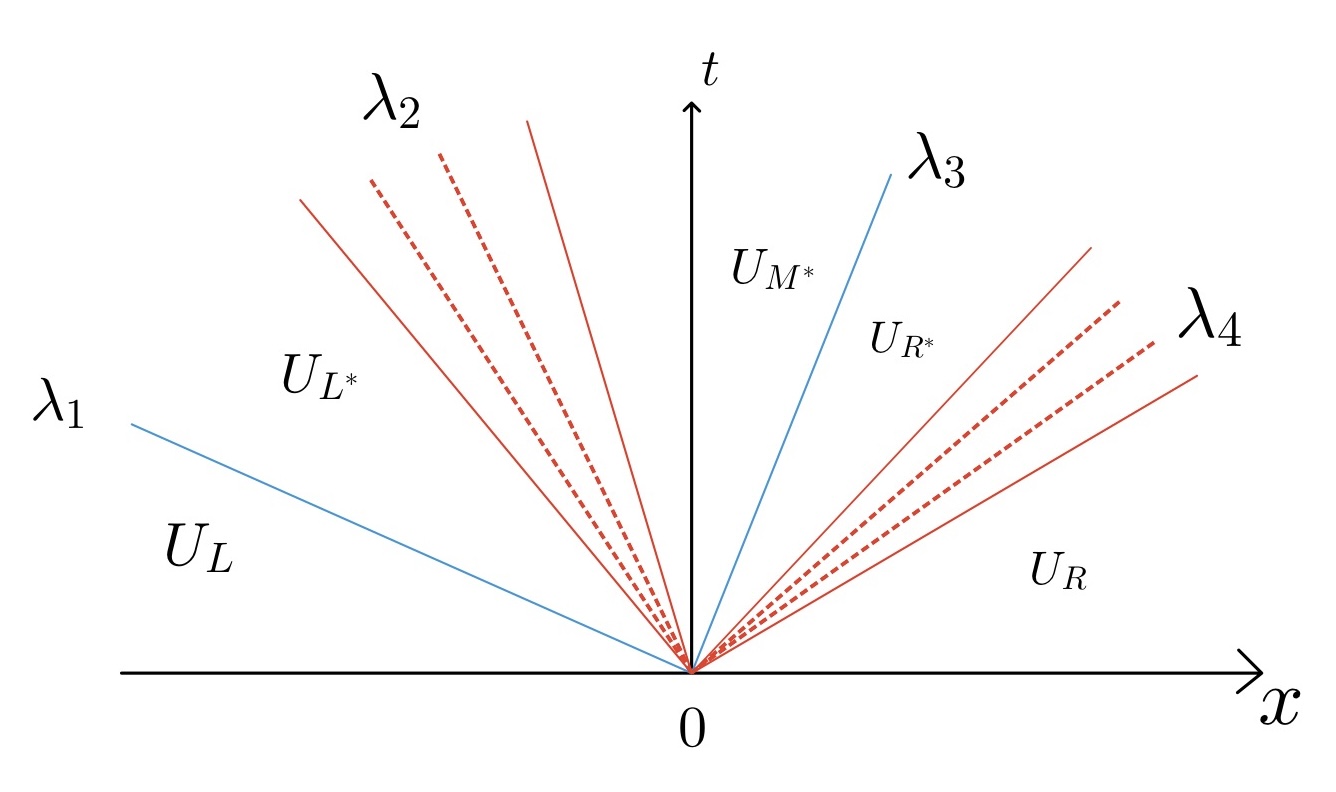}\\
  \caption{Typical wave configuration for\\
  Case 1(b). $J_1+R_2+J_3+R_4$.}\label{Case 1(b)}
\end{subfigure}%
\end{figure}
\subsubsection{Case 2}\label{subsubsec: 5.1.2}
Let us now consider the elementary wave curves from Section \ref{subsec: 4.1} for the proposed configuration $J_1+S_2+J_3+S_4$  (see \ref{RP-pic2}) or $J_1+S_2+J_3+R_4$ (see \ref{fig: case1(b)}). 
\begin{figure}
\captionsetup[subfigure]{font=scriptsize,labelfont=scriptsize}
\begin{subfigure}{0.5\textwidth}
  \centering
  \hspace{-1 cm}\includegraphics[width=2.2 in]{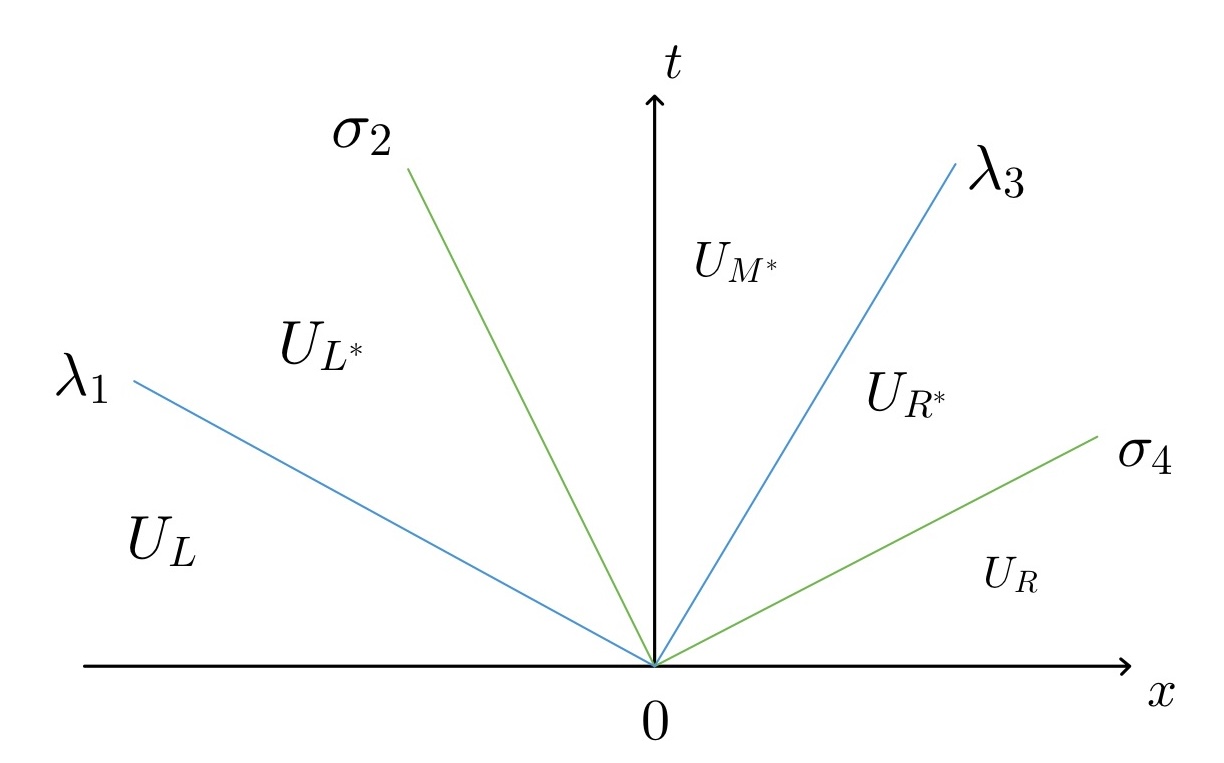}
  \caption{Typical wave configuration for\\
  Case 2(a). $J_1+S_2+J_3+S_4$.}
 \label{RP-pic2}
 \end{subfigure}%
 \begin{subfigure}{0.5\textwidth}
 \centering
\hspace{-1 cm}\includegraphics[width=2.2 in]{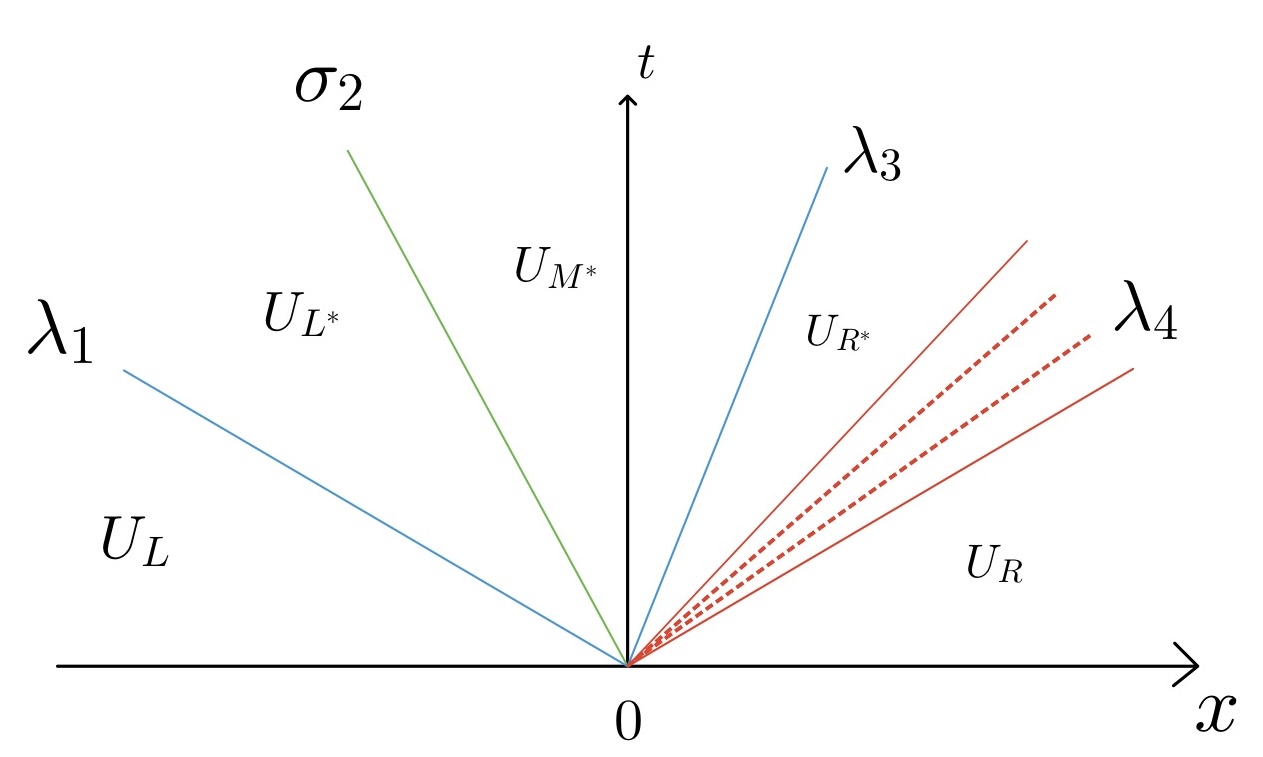}
  \caption{Typical wave configuration for\\
  Case 2(b). $J_1+S_2+J_3+R_4$.}
 \label{fig: case1(b)}
 \end{subfigure}
\end{figure}

For the first and third characteristic fields, we have contact discontinuities, and therefore we deduce from \eqref{eq: 23a}, \eqref{eq: 23b} the relations 
\begin{equation}\label{eq: 5.3J}
 g^*_L=g_L,\ q^*_L=q_L,\ f_Lb_L=f^*_Lb^*_L,\
 f^*_R=f_M^*,\ b^*_R=b_M^*,\  g_M^*q_M^*=g^*_Rq^*_R.
\end{equation}
From \eqref{S_2def}, we deduce for the $2$-shock wave 
\begin{align}\label{eq: 5.3S} 
&\hspace{2.9 cm} \ f_L^*/b_L^*=f_M^*/b_M^*,~\ g_L^*/q_L^*=g_M^*/q_M^*, ~\ \\
&2\sigma_2(\mathbf{U}_L^*, \mathbf{U}_M^*)(g_M^*-g_L^*)=(g_M^*)^2 q_M^*-(g_L^*)^2 q_L^*+2(f_M^*b_M^*g_M^*-f_L^*b_L^*g_L^*).
\end{align}
The relations for the 4-shock wave or 4-rarefaction wave from \eqref{eq: 24b} and \eqref{R4} imply
\begin{align}\label{eq: 5.3s}  
f_R^*=f_R,~b_R^*=b_R,~g_R^*/q_R^*=g_R/q_R.
\end{align}
Combining \eqref{eq: 5.3J}-\eqref{eq: 5.3s}, we observe that the intermediate states 
$\mathbf{U_L^*}, f^*_M, b^*_M, q^*_M, \mathbf{U_R^*} $ can be expressed in terms of the Riemann initial data and $g^*_M$ by the set of equations 
\begin{align*}
(f_L^*,~b_L^*,~g_L^*,~q_L^*)&=\left(\sqrt{\dfrac{f_Lb_Lf_R}{b_R}}, \sqrt{\dfrac{f_Lb_Lb_R}{f_R}}, g_L, q_L\right), \\
    f_M^* = f_R, & \quad  b_M^* = b_R, \quad  q_M^*=g_M^* \dfrac{q_L}{g_L},  \\
    (f_R^*,~b_R^*,~g_R^*,~q_R^*)&=\left(f_R,~ b_R, ~g_M^*\sqrt{\dfrac{q_Lg_R}{g_Lq_R}},~g_M^* \sqrt{\dfrac{q_Lq_R}{g_Lg_R}}\right).
\end{align*}
The unknown state $g_M^*$ is obtained by solving the nonlinear equation 
\begin{align}\label{eq: 5.3}
   F_2(g_M^*)&:= q_L(g_M^*)^3+g_M^* g_L(f_Rb_R-f_Lb_L-\sqrt{f_L b_L f_R b_R})\nonumber\\
   &\hspace{1 cm}-g_L^2(g_Lq_L+f_Lb_L-f_Rb_R-\sqrt{f_L b_L f_R b_R})=0.
\end{align}
The initial data condition $f_Rb_R<f_Lb_L$ implies $F_2(g_L)<0$ and $F_2(\infty)>0$. Therefore, the equation \eqref{eq: 5.3} always possesses a positive root in the interval $(g_L, \infty)$. Moreover, similar to the previous case, based on the values of $g_M^*$, we have the following two subcases:
\begin{enumerate}
    \item If $g_M^*>\sqrt{\dfrac{g_Lg_Rq_R}{q_L}}$ then the solution structure is $J_1+S_2+J_3+S_4$.
    \item If $g_M^*\leq \sqrt{\dfrac{g_Lg_Rq_R}{q_L}}$ then the solution structure is $J_1+S_2+J_3+R_4$.
\end{enumerate}
\begin{remark}
Here,  we consider generic $4$-wave solutions of the Riemann problem.  They include  $2$-wave solutions if we restrict ourselves to, e.g.,  one-layer dynamics. These persist 
under arbitrary variations of initial data in the remaining layer.   Solutions that consist of one (rarefaction) wave only are discussed in \cite{barthwal2025generalized}.      
\end{remark}
\section{Numerical experiments}\label{sec: 6}
In the last section, we derived the exact solution to the Riemann problem. We now proceed to some numerical experiments for particular test cases. We first consider the Riemann problems from Section \ref{subsubsec: 5.1.1} and \ref{subsubsec: 5.1.2}. Further, to gain more insights towards the solution to a general Cauchy problem for system \eqref{eq: Main_system}, we consider smooth initial data as well. In all our test cases for the Riemann problem, the computational domain is chosen to be $[-2, 12]$ with the final time of simulation as $t=1.00$. We impose zero Neumann boundary conditions. In particular, this choice ensures that no waves enter the computational domain through the boundaries.  We compare the exact solutions with solutions obtained using the Godunov and the Lax-Friedrichs scheme. For the smooth initial data, we take the interval $[-4,4]$ and periodic boundary conditions. 
\subsection{The Riemann problem for Case 1}


\begin{figure}
\centering
  \hspace{-0.7 cm}\includegraphics[width=4.5 in]{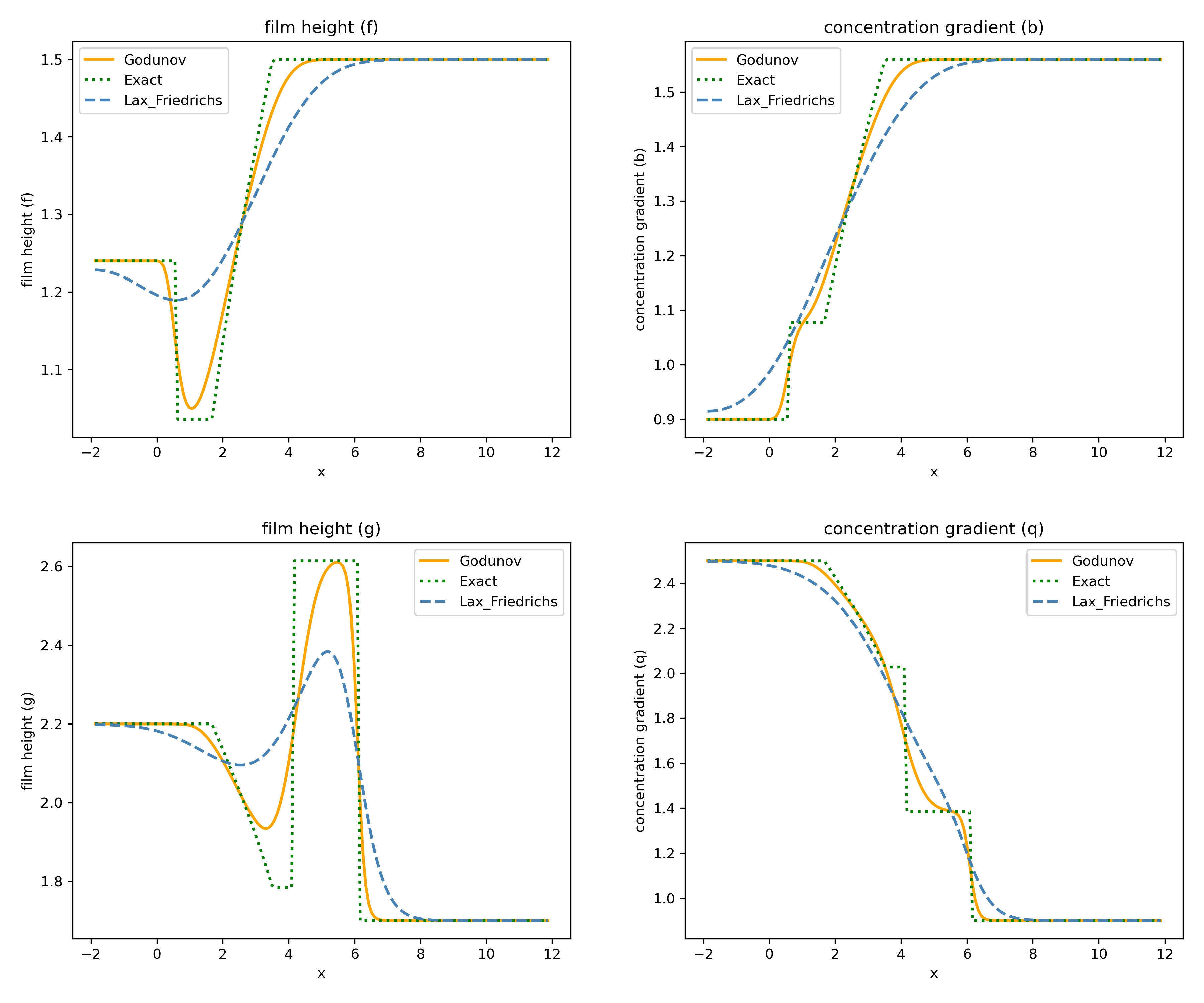}
  \caption{Case 1. Comparison 
  of numerical results for $\Delta x=8.75\times 10^{-2}$ at $t=1.00$.}
 \label{fig: case1_1}
 \end{figure}
 \begin{figure}
\centering
\hspace{-0.7 cm}\includegraphics[width=4.5 in]{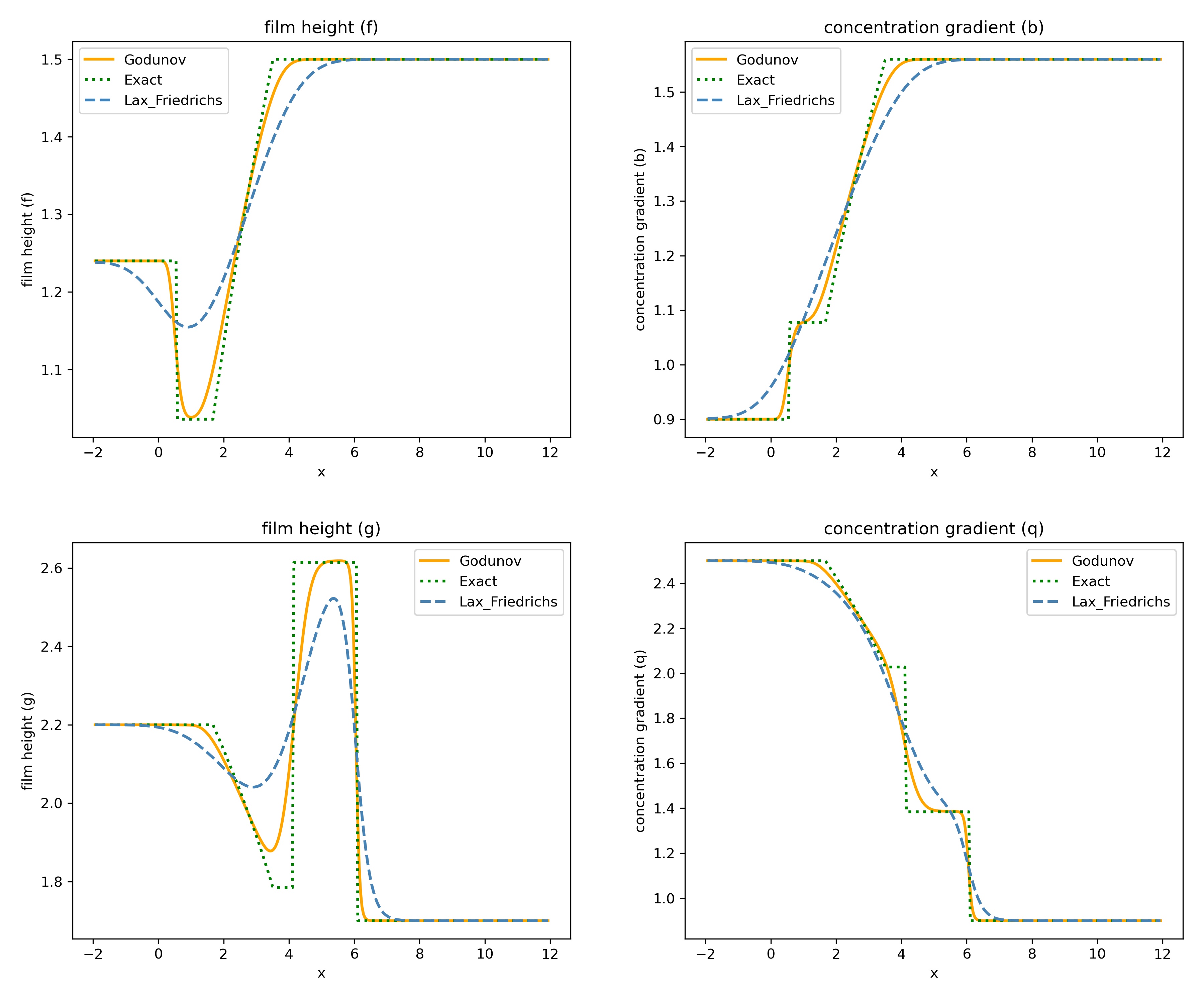}
  \caption{Case 1. Comparison of numerical results for $\Delta x=4.37\times 10^{-2}$ at $t=1.00$.}
\end{figure}
\begin{figure}
\centering
  \hspace{-0.5 cm}\includegraphics[width=4.5 in]{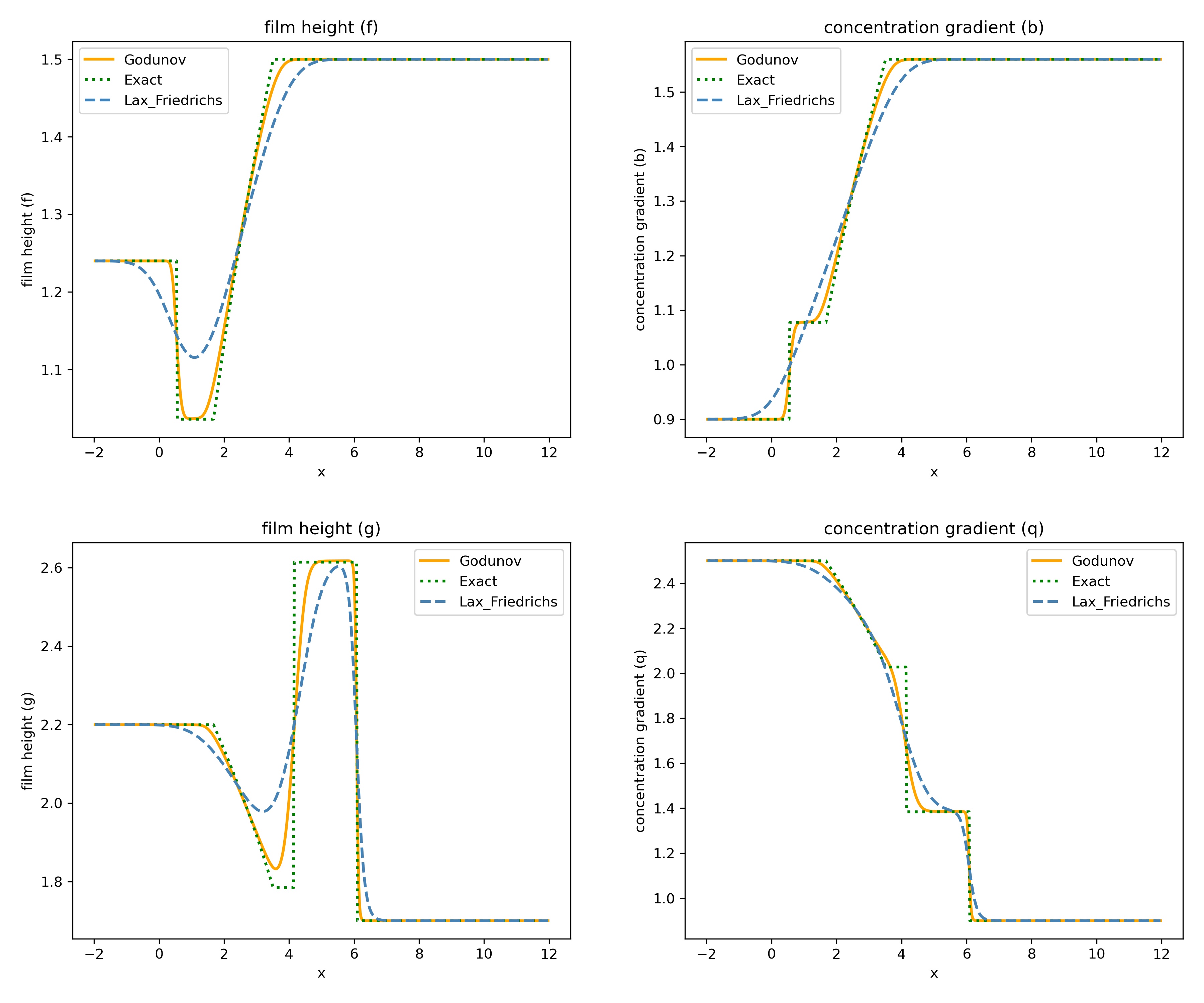}
  \caption{Case 1. Comparison
  of numerical results for $\Delta x=2.18\times 10^{-2}$ at $t=1.00$.}
  \label{fig: case1_2}
\end{figure}
For this test case, we choose our initial data of the form
\begin{align}
    \mathbf{U}(x, 0)= \left\{\begin{array}{rcl}
    (1.24, 0.90, 2.2, 2.50)^\top&: &~x<0, \\
    (1.5, 1.56, 1.7, 0.90)^\top &: &~x>0. 
    \end{array} \right.
\end{align}
Once the root of the nonlinear algebraic equation  \eqref{eq: 5.2} is obtained, one can produce the exact Riemann solution based on the relations discussed in Section \ref{sec: 5}. The solution structure for this case is given by $J_1+R_2+J_3+S_4$, see \ref{RP-pic}. We compare our solutions with the results of the  Lax-Friedrichs scheme. We choose the grid sizes to be $\Delta x= 
8.75\times 10^{-2},~4.37\times 10^{-2},$ and $ 2.18 \times 10^{-2}$. The numerical results suggest that for finer grid sizes, the Godunov scheme approximates the exact solution to the Riemann problem really well, while the Lax-Friedrichs scheme remains more diffusive, see Figures \ref{fig: case1_1}-\ref{fig: case1_2}. This is a typical behaviour since it does not rely on the exact Riemann solver like the Godunov scheme. Asymptotic convergence of the Godunov scheme and the Lax-Friedrichs scheme can also be observed from the numerical results.
\subsection{A shock-tube type initial data}
In this example, we consider the following initial data
\begin{align}
    \mathbf{U}(x, 0)= \left\{\begin{array}{rcl}
    (1.5, 1.6, 1.6, 2.00)^\top&: &~x<0, \\
    (1.25, 1.15, 2.0, 2.1)^\top &: &~x>0. 
    \end{array} \right.
\end{align}
The solution structure for this case is given by $J_1+S_2+J_3+S_4$ and consists only of discontinuities. We plot the numerical results for the grid size $\Delta x= 
4.37\times 10^{-2}$ in Figure \ref{fig: J+S+J}. One can observe that the Godunov solver captures the discontinuities much better.
\begin{figure}
    \centering
    \includegraphics[width=4.5 in]{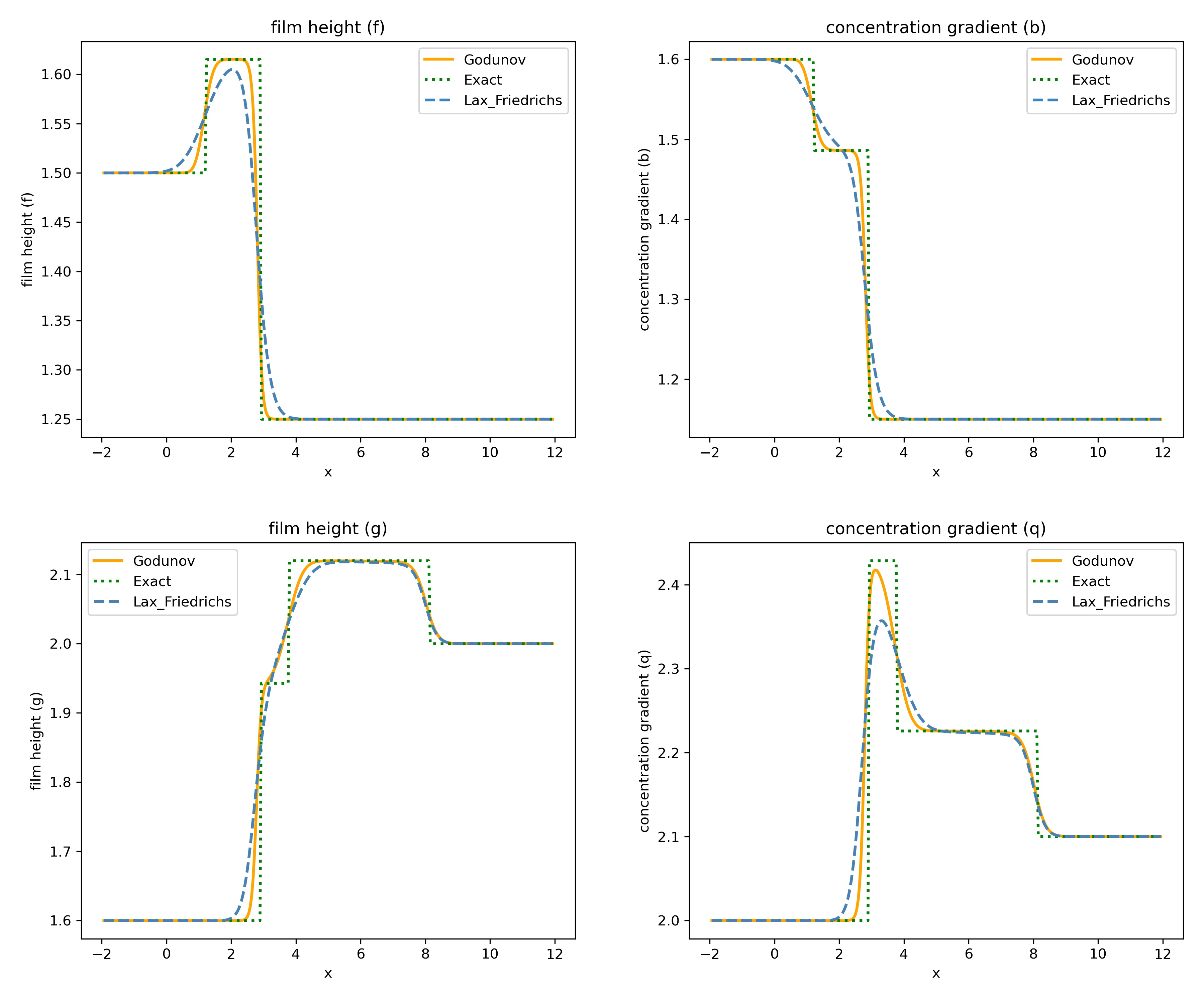}
    \caption{Comparison
  of numerical results for a shock tube type initial data with $\Delta x=2.18\times 10^{-2}$ at $t=1.00$.}
    \label{fig: J+S+J}
\end{figure}
\subsection{Solution for a smooth initial profile}
Till now, we have developed solutions to the Riemann problem for system \eqref{eq: Main_system} to validate our solver. Now we employ it to compute a numerical solution to the Cauchy problem for system \eqref{eq: Main_system} with smooth initial data of the form 
\begin{align*}
    f(x, 0)=g(x, 0)=0.2,~b(x, 0)=1.51+\kappa e^{-(x-0.05)^2},~q(x, 0)= 0.2+10e^{-(x-0.05)^2},
\end{align*}
where $\kappa$ is a parameter. 
The computational domain is $[-4, 4]$ and periodic boundary conditions are used here. The choice of this initial data suggests that initially, the concentration gradients in each layer follow a normal distribution, and the heights of the two thin films are constant. For this choice of initial data, no explicit exact solution is a priori known. However, we use the Godunov scheme to gain some insights into the evolution.  We vary $\kappa$ values as $20, \,30, \,40$ and $50$ and plot the numerical results at time $t=1.5$ with $\Delta x= 1.25\times 10^{-2}$ in Figure \ref{fig:sub2} to understand the effect of concentration gradients on film heights. We also provide a 3-D plot of the evolution for each variable using the Godunov scheme with $\Delta x= 2.5\times 10^{-2}$ and $\kappa=50$ in Figure \ref{fig: 3-D plot} at $t=1.5$. In the simulations, one can observe the emergence of sharp singularities (shock waves) in the evolution of the film heights as $\kappa$ increases.  Moreover, the inverse relation of height vs.$\!$ concentration gradient post shock is worth noticing. The concentration gradients decrease in the domains where the film heights increase and vice versa.  We refer also to the   discussion on the physical interpretation of the entropy $\bar E$ in Remark \ref{rem42}. 
The decrease of the heights can be observed especially in the evolution of the film height of the first layer.  The film height decreases  with bigger values of $\kappa$, which corresponds to high concentration gradients of the solute.  But a film rupture  in finite time is not observed for the system \eqref{eq: Main_system}.\\
This behaviour is also reflected in the Riemann solutions constructed in Section \ref{sec: Riemann}. In particular, the intermediate state $f_L^* \in \mathcal U$ for the film height in the first layer satisfies $f_L^*=\sqrt{({f_Lb_Lf_R})/{b_R}}$. This implies that $f_L^*\rightarrow 0$ as $b_R\rightarrow \infty$. However, this leads us to the boundary of the state space, where the system does not remain strictly hyperbolic.  To resolve this issue and to get rupture in finite time, one may need to consider measure-valued $\delta$-shock waves. Rupture can not be achieved within the class of classical waves. For the one-layer case, a more detailed analysis for this critical case using  $\delta$-shock waves is available; see e.g. \cite{barthwal2023construction, barthwal2025existence, sen2020delta}, and references cited therein.  
\begin{figure}
\centering
   \hspace{-0.8 cm} \includegraphics[width=5.5 in]{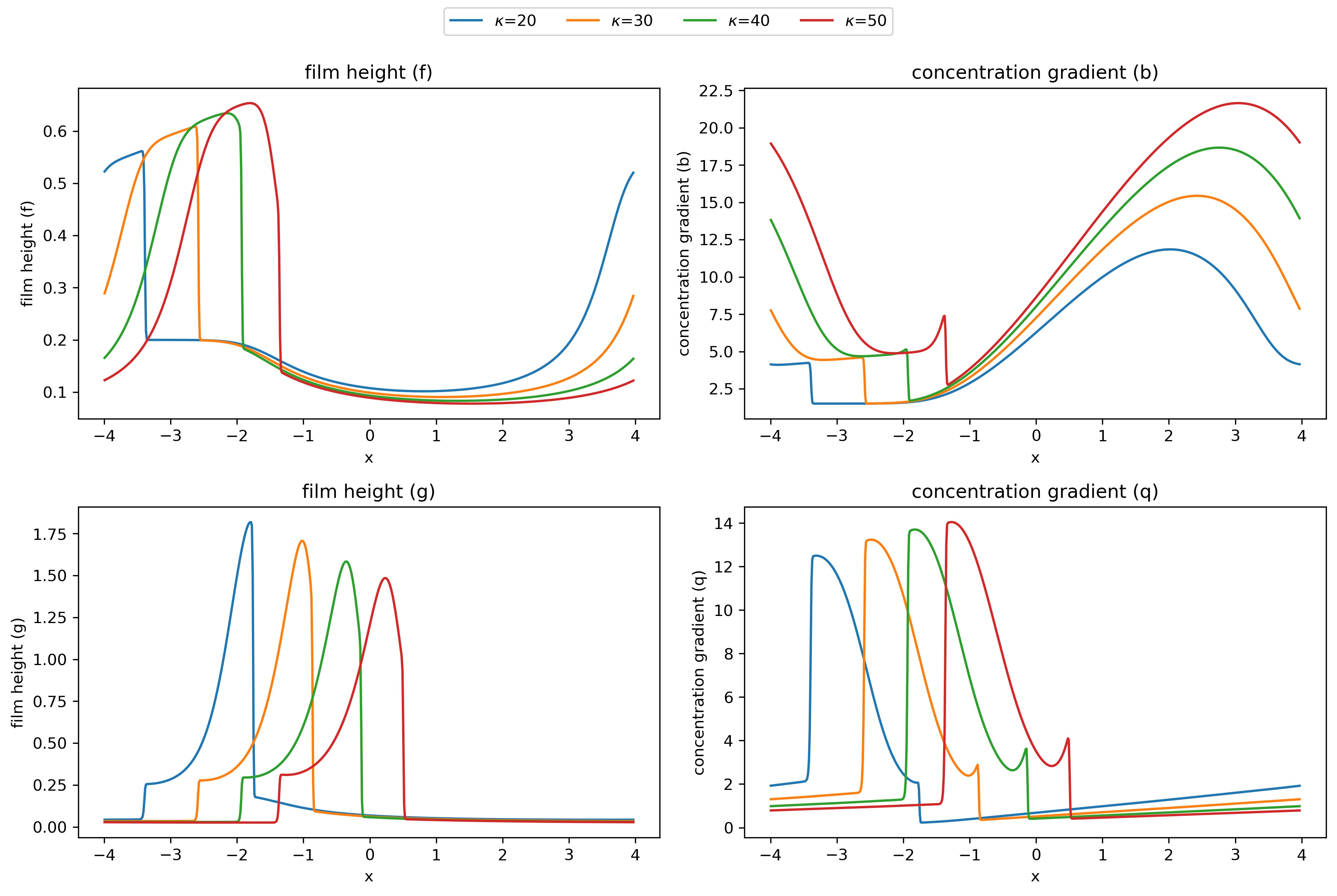}
  \caption{Numerical results 
  at $t=1.50$ using the Godunov scheme with $\Delta x=1.25\times 10^{-2}$.}
  \label{fig:sub2}
\end{figure}%
\begin{figure}
\centering
   \hspace{-0.5 cm}  \includegraphics[width=5.5 in]{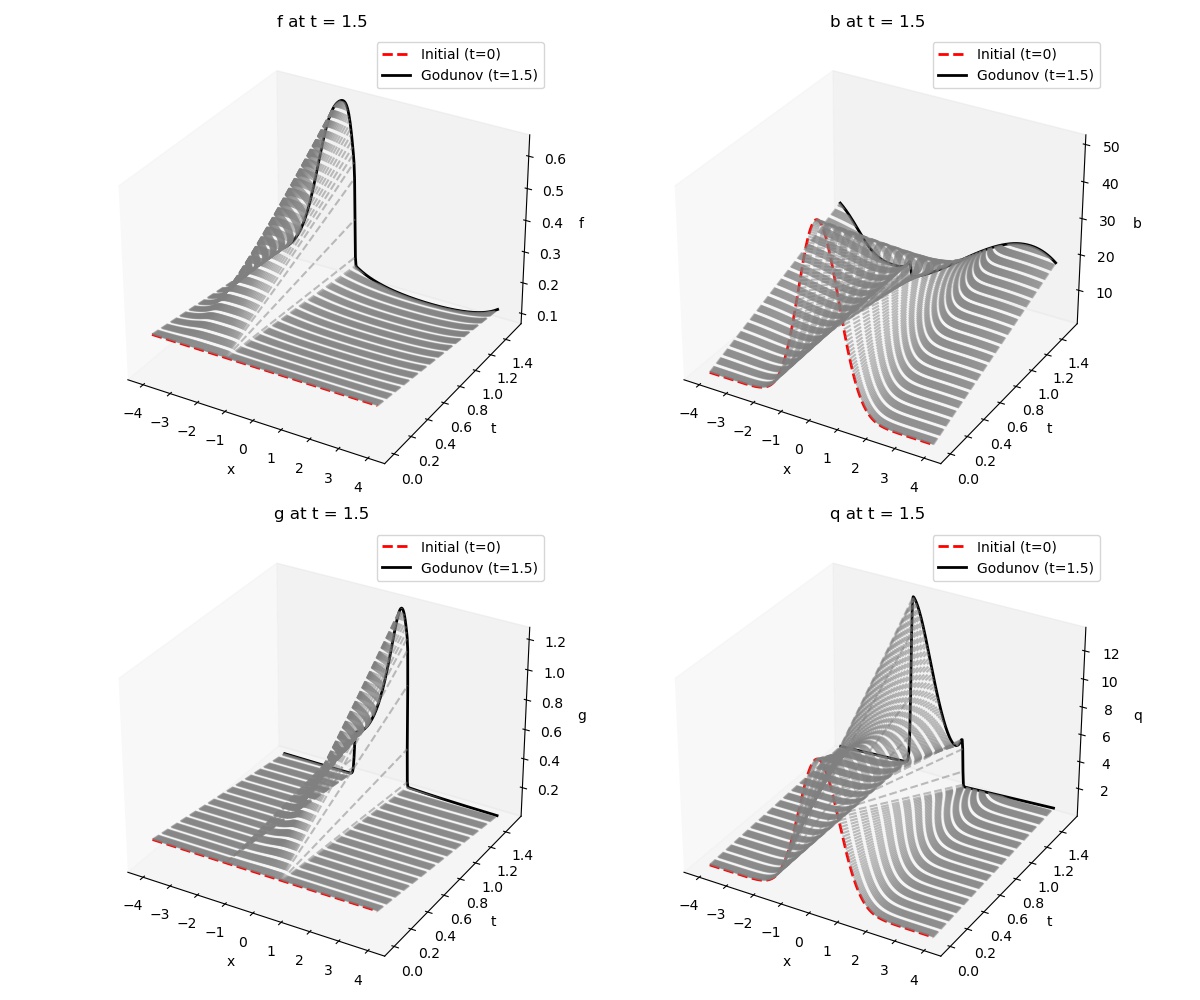}
    \caption{
    Solution profile using the Godunov scheme with $\Delta x= 2.5\times 10^{-2}$ at $t=1.5$.}
    \label{fig: 3-D plot}
\end{figure}


\section{Conclusions}\label{sec: 7}
In this work, we derived a reduced system of hyperbolic conservation laws governing the evolution of a two-layer thin film with a perfectly soluble anti-surfactant solute. The obtained system not only models surface-tension driven flow of a two-layer film but also exhibits a rich mathematical structure relevant to the theory of hyperbolic conservation laws. A crucial outcome of this work is the identification of a novel set of entropy/entropy-flux pairs, a set of Riemann invariants that forms a coordinate system, and an exact Riemann solver as the foundation for finite volume solvers. We emphasize that the results presented here were derived entirely within the  framework of classical waves for  hyperbolic conservation laws (shock, rarefaction, contact). This naturally raises the question of how the solution structure may evolve in more complex settings, particularly in regimes  on the boundary of the state space where strict hyperbolicity is lost. The analysis of such scenarios, which may involve intricate wave interactions and structural changes in the solution, is left for further research.

Moreover, our study opens several other new research directions. Extending the model to multi-layer thin film flows is possible. It would be interesting to see how the special mathematical structure we found for the two-layer case is recovered. The same applies if one accounts for gravity and flow on inclined surfaces \cite{barthwal2025existence}. Finally, we think that the found entropy/entropy-flux pairs are important to derive appropriate energy estimates for higher-order lubrication models, including 
capillarity and diffusion effects.
The development of entropy-stable numerical schemes, particularly high-order finite volume or discontinuous Galerkin methods, is a natural next step. Moreover, leveraging exact or approximate Riemann solvers, including Generalized Riemann solvers \cite{barthwal2025generalized}, would enhance the robustness of numerical approximations.
\medskip 
\vspace{-0.2 cm}\\\\
\textbf{Acknowledgements}~  This work was financially supported by the German Research Foundation (DFG), within the  Collaborative Research Center on Interface-Driven Multi-Field Processes in Porous Media (SFB 1313, Project Number 327154368) and the Priority Programme—SPP 2410 Hyperbolic Balance Laws in Fluid Mechanics: Complexity, Scales, Randomness (CoScaRa). The authors gratefully acknowledge the anonymous referee for their insightful suggestions and constructive feedback, which helped to enhance the presentation of the manuscript.\\
\appendix
\section{Derivation of velocity profiles}\label{sec: appendix}
To obtain the velocity components in each layer, we first solve the reduced system \eqref{eq: 2.3a}-\eqref{eq: 2.3e} with dimensionless boundary conditions \eqref{eq: 2.4a}-\eqref{eq: 2.4k2}. Now from \eqref{eq: 2.3d} and the normal stress balance \eqref{eq: 2.4e}-\eqref{eq: 2.4f}, it is easy to see that 
\begin{align*}
p_2=-\dfrac{1}{\mathrm{Ca_2}}\left(\dfrac{\partial^2 (f+g)}{\partial x^2}\right),~ \text{and}\quad p_1=-\dfrac{\mu}{\mathrm{Ca_2}}\left(\dfrac{\partial^2 (f+g)}{\partial x^2}\right)-\dfrac{1}{\mathrm{Ca_1}}\left(\dfrac{\partial^2 f}{\partial x^2}\right)
\end{align*}
hold. Hence, by \eqref{eq: 2.3b}, we have
\begin{align*}
    \dfrac{\partial u_1}{\partial z}=\bigg[-\dfrac{\mu}{\mathrm{Ca_2}}\left(\dfrac{\partial^3 (f+g)}{\partial x^3}\right)-\dfrac{1}{\mathrm{Ca_1}}\left(\dfrac{\partial^3 f}{\partial x^3}\right)\bigg]z+a_1. 
\end{align*}
Using the no-slip condition \eqref{eq: 2.4a}, we get
\begin{align}\label{eq: 2.22mainb}
   u_1=\bigg[-\dfrac{\mu}{\mathrm{Ca_2}}\left(\dfrac{\partial^3 (f+g)}{\partial x^3}\right)-\dfrac{1}{\mathrm{Ca_1}}\left(\dfrac{\partial^3 f}{\partial x^3}\right)\bigg]\dfrac{z^2}{2}+a_1z. 
\end{align}
Now from the tangential stress balance \eqref{eq: 2.4g}, we obtain
\begin{align*}
    a_1= \mu \dfrac{\partial u_2}{\partial z}\bigg|_{z=f}+\mathrm{Ma_1} \dfrac{\partial c_1}{\partial x}+\bigg[\dfrac{\mu}{\mathrm{Ca_2}}\left(\dfrac{\partial^3 (f+g)}{\partial x^3}\right)+\dfrac{1}{\mathrm{Ca_1}}\left(\dfrac{\partial^3 f}{\partial x^3}\right)\bigg]f.
    \end{align*}
From \eqref{eq: 2.3b}, we have $$\dfrac{\partial u_2}{\partial z}=-\dfrac{1}{\mathrm{Ca_2}}\left(\dfrac{\partial^3 (f+g)}{\partial x^3}\right)z+a_2.$$
Thus, using the tangential stress balance, we get
$$a_2= -\mathrm{Ma_2}\left(\dfrac{\partial c_2}{\partial x}+\dfrac{\partial c_2}{\partial z}\dfrac{\partial (f+g)}{\partial x}\right)+\dfrac{1}{\mathrm{\mathrm{Ca_2}}}\left(\dfrac{\partial^3 (f+g)}{\partial x^3}\right)(f+g)$$
or 
\begin{align}\label{eq: 2.20u_2}
\dfrac{\partial u_2}{\partial z}=\dfrac{1}{\mathrm{Ca_2}}\left(\dfrac{\partial^3 (f+g)}{\partial x^3}\right)((f+g)-z)+\mathrm{Ma_2}\left(\dfrac{\partial c_2}{\partial x}+\dfrac{\partial c_2}{\partial z}\dfrac{\partial (f+g)}{\partial x}\right),
\end{align}
which implies 
\begin{align*}
    \dfrac{\partial u_2}{\partial z}\bigg|_{z=f}=\dfrac{1}{\mathrm{Ca_2}}\left(\dfrac{\partial^3(f+g)}{\partial x^3}\right)g+\mathrm{Ma_2}\left(\dfrac{\partial c_2}{\partial x}+\dfrac{\partial c_2}{\partial z}\dfrac{\partial (f+g)}{\partial x}\right). 
\end{align*}
Hence
\begin{align*}
    a_1&=\bigg[\dfrac{\mu}{\mathrm{Ca_2}}\left(\dfrac{\partial^3 (f+g)}{\partial x^3}\right)(f+g)+\dfrac{1}{\mathrm{Ca_1}}\left(\dfrac{\partial^3 f}{\partial x^3}\right)f\bigg]+\mu \mathrm{Ma_2}\left(\dfrac{\partial c_2}{\partial x}+\dfrac{\partial c_2}{\partial z}\dfrac{\partial (f+g)}{\partial x}\right)+\mathrm{Ma_1} \dfrac{\partial c_1}{\partial x}.
\end{align*}
Therefore, by \eqref{eq: 2.22mainb}, we obtain
$u_1$ as defined in \eqref{eq: 2.18main}.

The expression of $u_2$ in \eqref{eq: 2.6a_velocity_2} can be obtained in a similar manner.
\bibliographystyle{abbrv}
\bibliography{references}
\end{document}